\documentclass[11pt]{amsart}
\usepackage{amsthm,amsmath,amsxtra,amscd,amssymb,xypic,cleveref,color}
\usepackage[all]{xy}

\usepackage{txfonts}
\usepackage{a4wide}

\newcommand\tc[2]{\theta\chars{#1}{#2}}
\newcommand\chars[2]{\left[\begin{smallmatrix}#1\\ #2\end{smallmatrix}\right]}

\newcommand{\CC}{{\mathbb{C}}}

\newcommand{\HH}{{\mathbb{H}}}

\newcommand{\PP}{{\mathbb{P}}}
\newcommand{\QQ}{{\mathbb{Q}}}
\newcommand{\RR}{{\mathbb{R}}}
\newcommand{\ZZ}{{\mathbb{Z}}}

\newcommand{\calO}{{\mathcal O}}

\newcommand{\calT}{{\mathcal T}}

\newcommand{\calA}{{\mathcal A}}

\newcommand{\calM}{{\mathcal M}}

\newcommand{\calX}{{\mathcal X}}
\newcommand{\calY}{{\mathcal Y}}

\newcommand{\op}{\operatorname}
\newcommand{\ab}[1][3]{\calA_{#1}}
\newcommand{\ua}[1][2]{\calX_{#1}}

\newcommand{\wub}[1][1]{\widetilde\calY_{#1}}
\newcommand{\wDelta}{\widetilde{\Delta}}
\newcommand{\oub}[1][1]{\overline{\calY}_{#1}}
\newcommand{\oa}[1][3]{\overline{\calA_{#1}}}
\newcommand{\ox}[1][2]{\overline{\calX_{#1}}}
\newcommand{\Sat}[1][3]{{\calA_{#1}^{\op {Sat}}}}
\newcommand{\Perf}[1][g]{{\calA_{#1}^{\op {Perf}}}}
\newcommand{\Vor}[1][g]{{\calA_{#1}^{\op {Vor}}}}
\newcommand{\Centr}[1][g]{{\calA_{#1}^{\op {Centr}}}}
\newcommand{\Sp}{\op{Sp}}

\newcommand{\PSL}{\op{PSL}}

\newcommand{\GL}{\op{GL}}
\newcommand{\SL}{\op{SL}}

\newcommand{\Exc}{\op{Exc}}
\newcommand{\Hyp}{\op{Hyp}}
\newcommand{\Sym}{\op{Sym}}

\newcommand{\Pic}{\op{Pic}}

\newcommand\Eff[1][2]{\operatorname{Eff}_{#1}}
\newcommand\Nef[1][2]{\operatorname{Nef}^{#1}}
\newcommand\oE[1][2]{\overline{\operatorname{Eff}}_{#1}}

\newcommand\si[1]{\sigma_{#1}}
\newcommand\osi[1]{\overline{\si{#1}}}

\theoremstyle{plain}
\newtheorem{thm}{Theorem}[section]
\newtheorem{lm}[thm]{Lemma}
\newtheorem{prop}[thm]{Proposition}
\newtheorem{cor}[thm]{Corollary}

\newtheorem{conj}[thm]{Conjecture}

\theoremstyle{definition}
\newtheorem{df}[thm]{Definition}
\newtheorem{rem}[thm]{Remark}

\title{On the cone of effective surfaces on $\oa$}
\author{Samuel Grushevsky}
\address{Mathematics Department, Stony Brook University,
Stony Brook, NY 11794-3651, USA}
\email{sam@math.stonybrook.edu}
\thanks{Research of the first author is supported in part by the National Science Foundation under the grant DMS-18-02116.}
\author{Klaus Hulek}
\address{Institut f\"ur Algebraische Geometrie, Leibniz Universit\"at Hannover, Welfengarten 1, 30060 Hannover, Germany}
\email{hulek@math.uni-hannover.de}
\thanks{Research of the second author is supported in part by DFG grant Hu-337/7-1.}

\begin{document}
\begin{abstract}
We determine five extremal effective rays of the four-di\-men\-sio\-nal cone of effective surfaces on the toroidal compactification $\oa$ of the moduli space $\ab$ of complex principally polarized abelian threefolds, and we conjecture that the cone of effective surfaces is generated by these surfaces.
As the surfaces we define can be defined in any genus $g\ge 3$, we further conjecture that they generate the cone of effective surfaces on the perfect cone compactification $\Perf[g]$ for any $g\ge 3$.
\end{abstract}

\maketitle

\section*{Introduction}
For any projective variety~$X$ the effective cone $\Eff[k](X)\subset H_{2k}(X,\RR)$  is  the cone generated  by classes of dimension~$k$ subvarieties.
The closure $\oE[k](X)$ is called the pseudoeffective cone. We recall that a pseudoeffective class $S\subset \oE[k](X)$ is called an {\em extremal ray} if the equality $S=S_1+S_2$,
with $S_1,S_2\in\oE[k](X)$, implies that $S_1$ and $S_2$ are multiples of~$S$.  Following~\cite{chcohighercodim}, we will call an effective class $S\in \Eff[k](X)$ an {\em extremal effective ray}  if the
equality $S=S_1+S_2$, with $S_1,S_2\in\Eff[k](X)$, implies that $S_1$ and $S_2$ are multiples of~$S$.

If~$X$ is the underlying variety of a  smooth $n$-dimensional DM stack, a situation we will typically be in, then $X$ satisfies Poincar\'e duality and codimension~$k$ subvarieties define classes in the cohomology group $H^{2k}(X,\RR)$.
The nef cone $\Nef[k](X)\subset H^{2k}(X,\RR)$  is then defined as the cone generated by the classes of effective subvarieties~$N$ of codimension $k$ such that for any $k$-dimensional subvariety $S\in \Eff[k](X)$ the pairing $\langle S, N\rangle$ is non-negative.

Motivated partly by the desire to know whether the canonical class $K_X$ is (pseudo)effective for a given~$X$, as this controls the Kodaira dimension, the study of the cone of effective divisors of various
varieties is a very rich and developed field. Algebraic curves on varieties provide a method to probe various geometric properties of the variety as well, and thus the cone of effective curves $\oE[1](X)$, dual to $\Nef[1](X)$, has also
been actively investigated for various varieties. One particularly fertile ground for such study has been moduli spaces. Since the foundational works of Mumford, the cone of effective divisors on the moduli space of stable curves has
 been actively investigated, with much progress made,  see the survey by Chen, Farkas, Morrison~\cite{chfamo}, and with Mullane~\cite{mullane} recently showing that the cone of effective divisors may not be finitely generated.
 On the other hand, the cone of nef divisors on the moduli of stable curves was studied for example by Gibney, Keel, and Morrison~\cite{gikemo}.

More recently, the cones of higher dimensional effective cycles on moduli of stable curves were studied by Chen and Coskun~\cite{chcohighercodim} who constructed not finitely generated effective cones. Furthermore, Mullane~\cite{mullane1} and Chen and Tarasca~\cite{chentarasca} constructed other extremal effective higher codimension classes, and Blankers~\cite{blankers} showed the extremality of all boundary strata that parameterize curves with rational tails, while Schaffler~\cite{schaffler} specifically studied the cone of effective surfaces on $\overline{\calM}_{0,7}$. In general, cones of higher dimensional (not divisorial) effective cycles have been attracting much attention recently, see for example
by~\cite{bastiaetal} and~\cite{coskunlesieutreottem}.

\smallskip
Our focus in this paper is the effective cycles on the compactifications of the moduli stack $\calA_g$ of (complex) principally polarized abelian varieties. The moduli stack $\calA_g$ is not compact and admits a minimal Satake-Baily-Borel compactification
$\Sat[g]$ as well as various toroidal compactifications. A fundamental result of Shepherd-Barron~\cite{shepherdbarron} is the determination of the cone of nef divisors of the perfect cone toroidal compactification $\Perf$ of~$\calA_g$, for any $g$.
Denoting by $L_g$ the class of the Hodge bundle on $\Perf$, and by~$D_g$ the class of the boundary divisor $D_g=\partial\Perf$ (which is irreducible for this toroidal compactification), Shepherd-Barron proves that $\Nef[1](\Perf)=\RR_{\ge 0}L_g+\RR_{\ge 0} M_g$, where from now on we denote $M_g:=12L_g-D_g$. Dually, this is equivalent to determining generators for the cone of effective curves $\oE[1](\Perf)$: Shepherd-Barron shows that it is generated by the classes of any curve $C_F\subset\partial\Perf$ that is contracted under the morphism $\pi:\Perf\to\Sat[g]$ to the Satake compactification and the class of the curve $C_A=\oa[1]\times [B]$, where $[B]\in\Perf[g-1]$ denotes any fixed point.

For arbitrary genus, the study of the cone of effective divisors on~$\Perf$ is a well developed subject, of independent interest as the question of the minimal slope of Siegel cusp forms. In genus 2 it is classically known
that the cone of effective divisors is generated by the decomposable locus  and the boundary $D_2$:
\begin{equation*}\label{equ:effsurfacesgenus2}
 \oE(\oa[2])= \RR_{\geq 0}\, \oa[1] \times \oa[1] + \RR_{\geq 0} D_2
\end{equation*}
where we have used the same notation for the product $\oa[1] \times \oa[1]$ and its image in $\oa[2]$ (see also the discussion below). In genus 3 the slope of the effective cone is equal to 9, and the cone of effective divisors is equal to $\oE[5](\Perf[3])=\RR_{\geq 0}D + \RR_{\geq 0} \Hyp_3$, where $\Hyp_3$ denotes the locus of hyperelliptic Jacobians. The cone of effective divisors is also known precisely for $\Perf[4]$ and $\Perf[5]$. The proof that the slope of the effective cone of $\Perf[6]$ is at most 7 is the breakthrough result of~\cite{sma6}, while for $g\to\infty$ the slope of the cone of effective divisors of $\Perf$ goes to zero by the results of Tai~\cite{tai}, see~\cite{grAgsurvey}.
For the central cone toroidal compactification $\Centr$ and for the second Voronoi toroidal compactification $\Vor$ of $\calA_g$ much less is known about the divisors, but the nef cone $\Nef[1](\Vor[4])$ was computed by the second author and Sankaran~\cite{husaA4}.

\smallskip
In this paper we make the first step towards exploring higher dimensional effective cycles on $\calA_g$ and its compactification. While, as mentioned above, already the situation for effective curves is unknown for $\Centr$ and $\Vor$ for high genus, the situation is special in genus $g\leq 3$ where the perfect cone, the second Voronoi, and the central cone compactifications of $\calA_g$
all coincide. For $g=2$ the effective surfaces are effective divisors, and both $\oE[1](\overline{\calA}_2)$ and $\oE[2](\overline{\calA}_2)$ are fully known, as discussed above. Our current interest will be the $g=3$ case, where we denote $\oa:=\Perf[3]=\calA_3^{\op {Vor}}=\calA_3^{\op {Cent}}$. We  recall that $\oa$ is a smooth DM stack where, furthermore, Chow and homological equivalence coincide, for the latter see \cite{huto1}.

The main result of our paper is to exhibit five extremal effective rays of the cone of effective surfaces on $\oa$.
To describe these surfaces  we recall that for any $0<i<g$ there is a finite morphism $N:\Perf[i]\times \Perf[g-i]\to\Perf[g]$, which is the normalization of the locus of decomposable abelian varieties (if $i \neq g-i$, otherwise it is the $2:1$ map
to the symmetric product).
Using the extremal rays of $\Eff[1](\oa[2])$, we define the following four surfaces in the product $\oa[1]\times\oa[2]$:
\begin{equation}\label{eq:surfaces12}
\begin{aligned}
S_{AA}:=\oa[1]\times C_A; \quad & S_{AF}:=\oa[1]\times C_F; \\
 S_{DD}:=[B]\times\partial\oa[2]; \quad & S_{DA}:=[B]\times N(\oa[1]\times\oa[1]).
\end{aligned}
\end{equation}
Here $[B]\in\oa[1]$ is a fixed point, and the motivation for the indexing is that the first subscript indicates whether we take all of $\oa[1]$ or its boundary (noting that
any point on $\oa[1]$ is equivalent to the cusp), and the second subscript indicates the kind of subvariety taken in the~$\oa[2]$ factor. This defines three surfaces in $\oa$:
\begin{itemize}
\item $S_A:=N(S_{AA})=N(S_{DA})$, pa\-rameter\-izing abelian three\-folds which are products of three elliptic curves one of which is fixed,
\item $S_F:=N(S_{AF})$, parameterizing abelian threefolds which are the product of some elliptic curve and a torus rank $1$ semi-abelic surface over a fixed elliptic curve.
\item $S_D:=N(S_{DD})$, parameterizing abelian threefolds which are the product of a fixed elliptic curve with a singular semi-abelic surface.
\end{itemize}

\begin{thm}[Main theorem]\label{thm:main}
The following five surfaces are extremal effective rays of $\Eff(\oa)$:
\begin{itemize}
\item the surfaces $S_A$, $S_F$ and $S_D$
\item the closures $\osi{K3+1}$ and $\osi{C4}$ of the two-dimensional boundary strata $\si{K3+1}$ and $\si{C4}$ of $\oa$.
\end{itemize}
\end{thm}
Here we recall that any toroidal compactification $\oa[g]$ is the union of certain strata which are indexed by the cones of the fan defining the toroidal compactification (modulo the action of the integral symplectic group).
For genus $3$ we have listed all equivalence classes of cones
in Table~\eqref{table:cones} below. We will denote the open stratum defined by a cone $C$  by $\si{C} \subset\oa$ and its closure by $\osi{C}$.

We stress that an extremal effective ray is {\em not} necessarily an extremal ray of the effective cone: it can a priori happen that an extremal effective ray~$S$ can be written as $S=S_1+S_2$, with at least one of the classes $S_i$ lying in $\oE[k](X)\setminus\Eff[k](X)$. In the current paper, as in~\cite{chcohighercodim}, we deal with extremal effective rays. See \Cref{rem:evidence} for the discussion of issues in trying to prove that the five surfaces above are in fact {\em extremal rays} of the pseudoeffective cone $\oE(\oa)$. We believe that proving our classes are in fact extremal rays would go a long way towards proving our main conjecture, that the classes above indeed  generate the effective cone of surfaces, and thus also the pseudoeffective cone:
\begin{conj}[Main conjecture]\label{conj:eff}
The pseudoeffective cone of surfaces on $\oa$ is the non-simplicial polyhedral cone given by
$$
 \oE(\oa)=\RR_{\ge 0}S_A+\RR_{\ge 0}S_F+\RR_{\ge 0}S_D+\RR_{\ge 0}\osi{C4}+\RR_{\ge 0}\osi{K3+1}\,.
$$
\end{conj}
Taking the image of each of the five surfaces described above under the embedding $\oa\to\Perf$ given by taking a product with a fixed $(g-3)$-dimensional ppav $[B]\in\Perf[g-3]$ exhibits our five effective surfaces as embedded in $\Perf$. It is known that $H^4(\Perf)=\RR^{\oplus 4}$ for all $g\ge 3$, with generators $L_g^2,L_gM_g, M_g^2$, and $\beta_2$, where $L_g$ and $M_g$ are divisor classes discussed above, and $\beta_2$ is the class of the closure of the stratum $\osi{1+1}$. It seems natural to us to also conjecture that the five surfaces above generate $\oE(\Perf)$ for {\em any} $g\ge 3$, see \Cref{rem:higherg} for more discussion.

\smallskip
To lighten the notation, from now on we drop the subscript~$g$ of classes~$L_g$ and~$M_g$ when the context is clear.
This main conjecture is equivalent to the following dual conjecture for the cone $\Nef(\oa)$ of nef codimension two cycles.
\begin{conj}\label{conj:nef}
The cone of nef  codimension two classes on~$\oa$ is equal to
$$
  \Nef(\oa)=\RR_{\ge 0} L^2+\RR_{\ge 0} LM+\RR_{\ge 0} M^2+\RR_{\ge 0} F_1+\RR_{\ge 0} F_2\,,
$$
where the classes $F_1$ and $F_2$ are given by
$$
F_1=-72\,L^2+12\,LM+3\,M^2+\beta_2;\qquad F_2=72\,L^2-8\,LM+M^2-\beta_2\,.
$$
\end{conj}

The main theorem can be interpreted dually as the statement that the nef cone $\Nef(\oa)$ is {\em contained} in the cone generated by the five classes above.
\begin{rem}\label{rem:knownnef}
We recall that by~\cite{shepherdbarron} the cone of nef {\em divisors} on the perfect cone compactification, in any genus, is generated by the classes $L_g$ and $M_g$. Since the product of nef divisors is a nef class, it follows that we have the inclusion
$$
  \Nef(\oa)\supset\RR_{\ge 0}L^2+\RR_{\ge 0}LM+\RR_{\ge 0}M^2\,,
$$
and thus the conjecture amounts to conjecturing the nefness of the classes $F_1$ and $F_2$. See \Cref{rem:evidence} for a discussion on more evidence for this conjecture.
\end{rem}

\begin{rem}\label{rem:higherg}
The main result of~\cite{grhuto} is that the cohomology of $\Perf$ stabilizes in close to top degree, i.e.~that $H^{g(g+1)-k}(\Perf)$ is independent of $g$ for any $g>k$, and then the natural speculation on $\oE(\Perf)$ being generated by the same five surfaces in all genera $g\ge 3$ would be equivalent to stabilization of $\Nef[2](\Perf)\subset H_{g(g+1)-4}(\Perf)$. While it is further tempting to speculate that in general the nef cone in a fixed codimension stabilizes, i.e.~that $\Nef[k](\Perf)$ is independent of $g$ for any $g>k$, it is not clear to us how to make sense of this statement: recall that in general $\Perf$ is highly singular, and it is not clear to us how to define the nef classes (since the codimension of the singular locus is at least 10 by~\cite{DutourHulekSchuermann}, this may still work for $k<10$). To approach the higher dimensional case it is natural to take the images of products of extremal effective classes under the maps $N:\oa[i]\times\oa[g-i]\to\oa[g]$, and in some situations some our methods can be used to show that such products are also extremal effective, see \Cref{rem:higherdim}.
\end{rem}

We note that much of our proof hinges on the knowledge of the homology classes of our five surfaces and the resulting computation of intersection numbers.
These homology class computations need to be done with extreme care: in particular, one has to use the correct factors resulting from the stack structure.
Throughout the paper, we thus work very carefully with the precise automorphism groups, and whenever possible we compute the classes of our surfaces in more than one way. A further check on our numbers is that we recompute in a different way some of the intersection numbers found by van der Geer \cite{vdgeerchowa3}.

\subsection*{Outline of the paper}
Here we give a brief outline of the structure of the paper.
Throughout this paper we will always consider (co)homology with real coefficients, and will always work with  DM stacks/orbifolds unless it is explicitly stated otherwise.
In \Cref{sec:stratification} we review the geometry and notation for the toroidal stratification of the boundary $\partial\oa$. In \Cref{sec:ox1} we quickly rederive the classically well-known intersection theory on the compactified universal elliptic curve
$\ox[1]\to\oa[1]$, which is the simplest model for our further computations. In \Cref{sec:oa1oa2} we use the known geometry of~$\oa[2]$ to describe the intersection theory on the product $\oa[1]\times\oa[2]$, which allows us to show in \Cref{prop:effoa1oa2} that $\oE(\oa[1]\times\oa[2])$ is generated by the (homology classes) of the surfaces $S_{AA}$, $S_{AF}$, $S_{DD}$ and $S_{DA}$. In \Cref{sec:classes} we compute the homology classes of their images, the surfaces~$S_A$,~$S_F$ and~$S_D$ in~$\oa$.
This will be essential for proving extremality of the surfaces~$S_A$ and~$S_F$, but will not suffice to prove extremality of~$S_D$.
Thus in \Cref{sec:ox1ox1} we study the intersection theory and in \Cref{prop:effV} determine the cone of effective surfaces of another auxiliary fourfold of independent interest, $V:=\ox[1]\times\ox[1]/S_2$. The usefulness of $\oa[1]\times\oa[2]$ and~$V$ is that they both naturally admit finite maps to subvarieties of $\oa$, which are in fact the natural subvarieties from the point of view of slope of modular forms on~$\oa$, see \Cref{rem:oa1oa2v}. In fact, we will identify $S_A,S_F$ and $S_D$ as images of surfaces contained in $V$, which will eventually allow us to prove extremality of $S_D$ and reprove extremality of $S_F$.

In \Cref{sec:osiK3} we undertake a detailed study of the closure of the three-dimensional toroidal boundary stratum $\osi{K3}\subset\oa$. This will not only be used to prove that our five surfaces are extremal effective, but is also of independent geometric interest.
It will also allow us to reconfirm the computation of some classes achieved in \Cref{sec:classes}, in particular of the class of $S_D$ and of  $\osi{K3+1}$.
Then in \Cref{sec:extremality} we prove that all five surfaces are extremal effective.
In \Cref{sec:SF} we construct another surface $S_P$ contained in the stratum $\osi{1+1}$, and study its interesting geometry.
As a result, we compute its homology class, which turns out to be equal to the class of~$S_F$.

Finally, in \Cref{sec:further} we make various remarks on open problems and further directions.

\subsection*{Acknowledgements}
The first author would like to thank Leibniz Universit\"at Hannover for hospitality during multiple visits in 2016-2020, when most of the work for this paper was done. He would also like to thank the Alexander von Humboldt foundation for the Bessel award which has greatly supported this cooperation. The second author is grateful to Stony Brook University for hospitality during stays in 2016 and 2020. We thank the referee for their careful reading of the manuscript.

\section{The stratification of $\oa$}\label{sec:stratification}
We recall the structure and the geometric description of the stratification of $\oa$, as described in~\cite{huto1} and~\cite{grhuto} --- the notation from where we fully adopt. Recall that there is a contracting morphism $\pi:\oa\to\Sat$. The Satake stratification is stratified as $\Sat=\ab\sqcup\ab[2]\sqcup\ab[1]\sqcup\ab[0]$, and we denote by $\beta_i^o:=\pi^{-1}(\ab[3-i])$ the corresponding stratification of $\oa$ into disjoint strata, and by $\beta_i:=\pi^{-1}(\Sat[3-i])=\beta_i^o\sqcup\beta_{i+1}^o\sqcup\dots\sqcup\beta_3$ the nested closed strata. There are two special features of genus 3. One is that each $\beta_i$ is irreducible and $\dim_\CC\beta_i=6-i$. The other is that on level covers $\oa(n)$ of level $n \geq 3$ any boundary component $D_i(n)$
is naturally fibred over $\oa[2](n)$ identifying $D_i(n) \cong \ox[2](n) \to \oa[2](n)$ with a compactification of the  universal family $\ua[2](n) \to \ab[2](n)$
with semi-abelic surfaces, see \cite{tsushima}, \cite{huleknef}.

Furthermore, each stratum $\beta_i^o$ can be represented as a disjoint union $\beta_i^o=\sqcup\beta(\sigma)$ of toroidal strata, where $\sigma$ ranges over all the rank $i$ cones in the corresponding toroidal fan, and $\beta(\sigma)$ denotes the stratum corresponding to such a cone. The strata decompose as follows: $\beta_1^o=\beta(\sigma_1)$; $\beta_2^o=\beta(\sigma_{1+1})\sqcup\beta(\sigma_{K3})$ and $\beta_3^o=\beta(\si{1+1+1})\sqcup\beta(\si{K3+1})\sqcup\beta(\si{C4})\sqcup\beta(\si{K4-1})\sqcup\beta(\si{K4})$,
where the cones and the dimensions of the corresponding strata are given as follows:
\begin{equation}\label{table:cones}
\begin{array}{l|c|l}
{\rm cone}&{\rm stratum\ dimension}&{\rm cone\ generators}\\[0.5ex]
\hline&&\\[-2ex]
0&6&\\[0.5ex]
\si{1}&5&x_1^2\\[0.5ex]
\si{1+1}&4&x_1^2,x_2^2\\[0.5ex]
\si{K3}&3&x_1^2,x_2^2,(x_1-x_2)^2\\[0.5ex]
\si{1+1+1}&3&x_1^2,x_2^2,x_3^2\\[0.5ex]
\si{K3+1}&2&x_1^2,x_2^2,x_3^2,(x_1-x_2)^2\\[0.5ex]
\si{C4}&2&x_1^2,x_2^2,(x_1-x_3)^2,(x_2-x_3)^2\\[0.5ex]
\si{K4-1}&1&x_1^2,x_2^2,x_3^2,(x_1-x_2)^2,(x_1-x_3)^2\\[0.5ex]
\si{K4}&0&x_1^2,x_2^2,x_3^2,(x_1-x_2)^2,(x_1-x_3)^2,(x_2-x_3)^2
\end{array}
\end{equation}
To keep the notation in the current paper manageable, we will abuse notation and write $\sigma$ for the cone of a decomposition giving an admissible fan, and also for the associated stratum~$\beta(\sigma)\subset\oa$.

As we are interested in complete surfaces, we will focus on the closures of the strata, which we then denote by $\osi{}$.
We will also denote the homology classes of these compactified strata by the same symbol. The containment of a given stratum in the closure of another corresponds to the (reversed)  containment of the corresponding cones (up to a suitable change of basis). In genus 3 the closures of the relevant strata are as follows:
\begin{equation}\label{eq:strataclosures}
\begin{aligned}\osi{1+1}=\beta_2&=\si{1+1}\sqcup\si{K3}\sqcup\si{1+1+1}\sqcup\si{K3+1}\sqcup\si{C4}\sqcup\si{K4-1}\sqcup\si{K4}\\
&=\si{1+1} \sqcup(\osi{K3}\cup\osi{1+1+1})\\
\osi{K3}&=\si{K3}\sqcup\si{K3+1}\sqcup\si{K4-1}\sqcup\si{K4}=\si{K3}\sqcup\osi{K3+1}\\ \osi{1+1+1}=\beta_3&=\si{1+1+1}\sqcup\si{K3+1}\sqcup\si{C4}\sqcup\si{K4-1}\sqcup\si{K4}\\
&=\si{1+1+1}\sqcup (\osi{K3+1}\cup\osi{C4})\,.
\end{aligned}
\end{equation}
As we are working with surfaces, the one-dimensional stratum $\si{K4-1}$ and the point $\si{K4}$ will play essentially no role in our study.

\section{The intersection theory on $\ox[1]$}\label{sec:ox1}
In what follows we shall need results on the intersection theory of $\oa$ and some of its subvarieties. We recall that the Chow ring of $\oa$ was computed by van der Geer in~\cite{vdgeerchowa3} and that the Chow and homology rings of $\oa$ are equal by~\cite{huto1}. It will be important for us to be very systematic when it comes to the stacky interpretation  of various intersection numbers we compute.

To start, and to illustrate the method, we shall compute the intersection theory of the universal elliptic curve $\ox[1]\to\oa[1]$. This is in some sense classical although difficult to trace in the literature, see e.g.  \cite{barthhulek}. Our presentation here  also serves as an illustration how intersection numbers are defined and derived: we work on a suitable level structure and then divide by the order of the group of deck transformations (in the sense of a DM stack).
The cohomology group $H^2(\ox[1])$ is $2$-dimensional, spanned by the Hodge line bundle $L=L_1$ and the (closure of) the $0$-section, which we will denote by~$Z$ (see, for example,~\cite{ergrhu1}). The (trivialized along the $0$-section) theta divisor~$T$ can be defined as the divisor which has the same degree on the fibers as $Z$, and whose restriction to~$Z$ has degree $0$.

To compute the intersection numbers, we will work with the level covers $\ox[1](n)\to\oa[1](n)$, which are also known as Shioda modular surfaces (often also denoted by $S(n) \to X(n)$). For a construction of these surfaces and their properties
we refer the reader to
~\cite{shioda}, ~\cite{barthhulek} and~\cite[Sec.~I]{hukawebook}. Here we recall the relevant properties of these surfaces.
By definition $\ox[1](n)\to\oa[1](n)$ is an elliptic surface over the modular curve $\oa[1](n)$. It has singular fibers over the cusps of $\oa[1](n)$, which are of type $I_n$, i.e.~they are $n$-gons of smooth rational curves, the self-intersection of each of which is equal to $-2$.

The deck group of the cover $\ox[1](n)\to\ox[1]$ is
\begin{equation*}
G(n)= \SL(2, \ZZ/n\ZZ) \ltimes (\ZZ/n\ZZ \times \ZZ/n\ZZ)\,.
\end{equation*}
Here $(\ZZ/n\ZZ \times \ZZ/n\ZZ)$ acts by addition of $n$-torsion points on the fibers, $-1 \in \SL(2, \ZZ/n\ZZ)$ acts as the Kummer involution, and the projective group
$\PSL(2,\ZZ/n\ZZ)$ acts effectively on $\oa[1](n)$, with the quotient~$\oa[1]$ (as a variety).

The Mordell-Weil group of sections is (non-canonically) isomorphic to $\ZZ/n\ZZ \times ZZ/n\ZZ $. On each smooth fiber these sections cut out the group of $n$-torsion points. This is also true for the $n$-gons, since the smooth locus of such an
$n$-gon has a group structure isomorphic to $\mu_n \times \ZZ/n\ZZ$ where $\mu_n$ denotes the group of $n$-th roots of unity. In particular, every component of a singular fiber intersects $n$ of these sections (in smooth points of the $n$-gon).
We denote the sections by $L_{ij}, (i,j) \in  \ZZ/n\ZZ \times \ZZ/n\ZZ$ and write
\begin{equation}\label{equ:MWgroup}
Z_n=\sum_{(i,j) \in \ZZ/n\ZZ \times \ZZ/n\ZZ} L_{ij}\,.
\end{equation}
The quotient of the union of these sections under $G(n)$ is the $0$-section $Z_1=Z$, where the lower index indicates the level cover on which we work. We finally note
that the group $\ZZ/n\ZZ \times \ZZ/n\ZZ$ acts on the $n$-gons as follows:  there is a subgroup $\ZZ/n\ZZ$ which permutes the components of the $n$-gon cyclically and another subgroup  $\ZZ/n\ZZ$ which acts on each component by
multiplication by a primitive $n$-th root of unity. This follows easily from the construction of the Shioda modular surfaces, see e.g.~\cite{barthhulek},\cite[Sec.~I]{hukawebook} .

The main numerical data which we will need are the following, see also e.g.~\cite[Sec.~I]{barthhulek}.
Let~$t(n)$ be the number of cusps of $\oa[1](n)$. Then
\begin{equation}\label{equ:numberfibers}
t(n)= \frac12 n^2 \prod_{p|n}(1 -{\frac{1}{p^2}})\,.
\end{equation}
The order of the Galois group is given by
\begin{equation}\label{equ:ordersl2z}
|\SL(2, \ZZ/n\ZZ)|=n^3 \prod_{p|n}(1 - \frac{1}{p^2})=2n\,t(n)\,,
\end{equation}
where the products are taken over all primes $p$ dividing $n$, and hence
\begin{equation}\label{equ:orderGaloisgroup}
|G(n)|=n^2\,|\SL(2, \ZZ/n\ZZ) |=2n^3\,t(n)\,.
\end{equation}
We finally note that it  follows from the canonical bundle formula and adjunction, see~\cite[p. 80]{barthhulek} that
\begin{equation}\label{equ:selfintersectionL}
L_{ij}^2= -\frac{n}{12}\,t(n)\,.
\end{equation}

Finally we recall the Hodge line bundle $L$. This is a $\QQ$-line bundle on $\oa[1]$.  Its pullback $L_n$ to a level~$n$ cover is an honest line bundle
for all $n\geq 3$.

We are now ready to start the calculation of intersection numbers on~$\ox[1]$ and related varieties.
\begin{prop}\label{prop:intersectionX1}
The classes $L,T,$ and $Z$ intersect as follows on $\ox[1]$:
\begin{equation}\label{table:intersectionX(1)}
  \begin{array}{r|ccc}
  &L&T&Z\\
  \hline\\[-2ex]
  L&0&\frac{1}{24}&\frac{1}{24}\\[0.5ex]
  T&\frac{1}{24}&\frac{1}{24}&0\\[0.5ex]
  Z&\frac{1}{24}&0&-\frac{1}{24}\,.
  \end{array}
\end{equation}
Moreover $T=Z+L$.
\end{prop}
\begin{proof}
We start with the computation of $Z^2$. By~\eqref{equ:MWgroup}, \eqref{equ:ordersl2z}, \eqref{equ:orderGaloisgroup}, and \eqref{equ:selfintersectionL} we find that
\begin{equation}\label{equ:selfintersetionZ}
Z^2=\frac{1}{|G(n)|}\left(\sum_{ij}L_{ij}\right)^2= - \frac{1}{24}\,.
\end{equation}

Next we discuss the class of the fiber of $\ox[1]\to\oa[1]$. For this let $F^0_n$ be a single fiber of $\ox[1](n)\to\oa[1](n)$. Let
\begin{equation}
F_n= \sum_{\bar{g} \in \PSL(2,\ZZ/n\ZZ)} \bar{g}(F^0_n)\,.
\end{equation}
Note that $G(n)$ acts on $F_n$ where $-1 \in \SL(2,\ZZ/n\ZZ)$ acts as the Kummer involution and $\ZZ/n\ZZ \times\ZZ /n\ZZ$ acts by addition of $n$-torsion points. We then find
\begin{equation}\label{equ:degreeZonF}
ZF=\frac{1}{|G(n)|} (Z_n.F_n)= \frac{1}{|G(n)|} |\PSL(2,\ZZ/n\ZZ)|\, n^2 =\frac12\,.
\end{equation}
Geometrically, the fraction $1/2$  is explained by the Kummer involution.

There is a unique cusp form
$\Delta \in H^0(\oa[1],12L)$ of weight $12$ whose pullback to  $\ox[1]$ vanishes to order~$1$ on the fiber over the cusp. The map $\oa[1](n) \to \oa[1](1)$ is branched of order $n$ along the cusps. Hence
we find that on $\ox[1](n)$ the equality
\begin{equation}
12L_n= n F_{\infty,n}
\end{equation}
holds, where $F_{\infty,n} $ is the sum of all fibers over all the cusps of $\oa[1](n)$. As there are $t(n)$ of these by~\eqref{equ:numberfibers}, we find that
\begin{equation}
12L_n=F_n
\end{equation}
and hence
\begin{equation}
12L=F\,.
\end{equation}
Together with~\eqref{equ:degreeZonF} this implies
\begin{equation} \label{equ:degreeZF}
LZ= \frac{1}{24}\,.
\end{equation}
In order to determine the class of $T$ we write $T=aZ+bL$. From the defining property $TZ=0$ and~\eqref{equ:selfintersetionZ},~\eqref{equ:degreeZF} we obtain
\begin{equation}
0=TZ=(aZ + bL)Z= - \frac{1}{24} a + \frac{1}{24} b
\end{equation}
and hence $a=b$. Finally since $Z.F=1/2$ by~\eqref{equ:degreeZF} and since $F^2=0$ as distinct fibers are disjoint, we find
\begin{equation}
\frac12 = ZF=TF=(aZ + bF)F= \frac12 a\,,
\end{equation}
and thus $a=b=1$. This shows
\begin{equation}
T=Z+L\,.
\end{equation}
Finally, a straightforward calculation using~\eqref{equ:selfintersetionZ} and~\eqref{equ:degreeZF} gives
\begin{equation}
T^2=(Z+L)^2= \frac{1}{24}
\end{equation}
and this concludes the proof.
\end{proof}

\section{The intersection theory and extremal surfaces on $\oa[1]\times\oa[2]$}\label{sec:oa1oa2}
In this section we compute the intersection theory on the product $\oa[1]\times\oa[2]$ and determine the cone
$\oE(\oa[1]\times\oa[2])$.

Denoting in this section by $L_2, D_2,$ and $M_2=12L_2-D_2$ the classes of the corresponding line bundles on~$\oa[2]$, we recall the classically known geometry in genus 2, and the intersection theory on~$\oa[2]$ as described by van der Geer \cite[Thm.~2]{vdgeerchowa3}. Indeed, $H^2(\oa[2])=\RR L_2\oplus \RR D_2$, and these classes generate the Chow (and homology) ring of $\oa[2]$ with the only relations being $L_2^2D_2=0$ and $(12L_2-D_2)(10L_2-D_2)=0$, which in terms of $M_2$ can be written as
$$
 M_2^2=2 L_2 M_2\quad{\rm and}\quad 12 L_2^3=L_2^2M_2.
$$
In particular, we see that $H^4(\oa[2])=\RR L_2^2\oplus\RR M_2^2$.
Since $\dim_\CC\oa[2]=3$, surfaces on $\oa[2]$ are divisors, and the effective cone is well understood:
$$
  \oE(\oa[2])=\RR_{\ge 0}\partial \oa[2]+\RR_{\ge 0} N(\oa[1]\times\oa[1]).
$$
Here $N: \oa[1] \times \oa[1] \to \oa[2]$ is the natural $2:1$ map onto the decomposable locus given by interchanging the factors.
The fundamental class of a surface on $\oa[2]$ can, by Poincar\'e duality,  be thought of as a class in $H_4(\oa[2])=H^4(\oa[2])^*=H^2(\oa[2])$ and the classes of the above surfaces are then identified as $[\partial\oa[2]]=D_2=12L_2-M_2$ while  we have $[N(\oa[1]\times\oa[1])]=5L_2-\frac12 D_2=\frac12(M_2-L_2)$ (see, for example, \cite[Lemma 2.2]{vdgeerchowa3}).
To understand this formula we recall that there is a weight 10 form which vanishes on a level-$n$ cover with multiplicity 1  on the locus of decomposable surfaces
and  with multiplicity $n$ on the boundary. The factor $1/2$ then comes from the fact that there is an involution interchanging the two factors when there is no level structure.

The cone of nef divisors is $\Nef[1](\oa[2])=\RR_{\ge 0}L_2+\RR_{\ge 0}M_2$, which is the special case of Shepherd-Barron's result in genus 2. Thus the cone of effective {\em curves} $\oE[1](\oa[2])$, which by definition is dual to $\Nef[1](\oa[2])$, is spanned by the two effective curve classes that have zero pairings with $L_2$ and $M_2$, respectively. Indeed, we have used these curves before:
firstly, there is the curve $C_F$ that is a generic fiber of $\partial\oa[2]\to\oa[1]$. Its class is $[C_F]=12L_2D_2$. For this we recall that geometrically $D_2=\partial\oa[2]$, and that $L_2|_{\partial\oa[2]}$ is the pullback of $L_1$ and that the class of a point on $\oa[1]$ is $12L_1$. Secondly, we have the curve $C_A=[B]\times\oa[1]$, whose class is $[C_A]=12L_2(10L_2-D_2)$.
The reason that here we have $10L_2-D_2$, which is twice the class of $N(\oa[1]\times \oa[1])$, is that the involution interchanging the two factors does not leave $[B] \times \oa[1]$ fixed.
As a check one immediately computes that $L_2C_A= 1/24$, which is the degree of the Hodge line bundle on $\oa[1]$.

\smallskip
Now recall that $H^2(\oa[1])=\RR L_1$; we thus compute the cohomology of the product
$$
  H^4(\oa[1]\times\oa[2])=\RR L_1L_2\oplus\RR L_1M_2\oplus\RR L_2^2\oplus\RR M_2^2,
$$
where by abuse of notation we use $L_1,L_2,M_2$ to denote the pullbacks of the corresponding classes on one of the factors to the product $\oa[1]\times\oa[2]$.

Thinking dually, to construct effective surfaces on $\oa[1]\times\oa[2]$, one can take an effective surface in $\oa[2]$, and take its product with a fixed point $[B]\in\oa[1]$, or one can take an effective curve in $\oa[2]$, and take its product with $\oa[1]$ (which is the only effective curve on $\oa[1]$). In this way, starting with the generators of $\oE(\oa[2])$ and $\oE[1](\oa[2])$, we construct the following four natural effective surfaces on~$\oa[1]\times\oa[2]$:
$$
\begin{array}{ll}
  S_{DD}:=[B]\times\partial\oa[2];\qquad& S_{DA}:=[B]\times N(\oa[1]\times\oa[1])\\
  S_{AF}:=\oa[1]\times C_F;\qquad & S_{AA}:=\oa[1]\times C_A.
\end{array}
$$
We claim that these four surfaces generate the pseudoeffective cone.
\begin{prop}\label{prop:effoa1oa2}
The cone of nef codimension two classes on $\oa[1]\times\oa[2]$ is generated by the four generators identified above:
$$
 \Nef(\oa[1]\times\oa[2])=\RR_{\ge 0} L_1L_2+ \RR_{\ge 0} L_1M_2+ \RR_{\ge 0}L_2^2+\RR_{\ge 0}M_2^2,
$$
while the cones of pseudoeffective and effective surfaces on $\oa[1]\times\oa[2]$ coincide, and are equal to
$$
 \oE(\oa[1]\times\oa[2])=\RR_{\ge 0} S_{DD}+ \RR_{\ge 0} S_{DA}+ \RR_{\ge 0}S_{AF}+\RR_{\ge 0}S_{AA}.
$$
\end{prop}
The Proposition will follow immediately from the following lemma computing the intersection numbers of the claimed generators of the nef and pseudoeffective cones, since this matrix turns out to be diagonal.
\begin{lm}\label{lm:intersectionsoa1oa2}
The intersection numbers of the nef generators of $H^4(\oa[2])$ with the four effective surfaces on $\oa[1]\times\oa[2]$ constructed above are as follows:
$$
  \begin{array}{r|cccc}
  &L_1L_2&L_1M_2&L_2^2&M_2^2\\
  \hline\\[-2ex]
  S_{AA}&\frac{1}{1152}&0&0&0\\[0.5ex]
  S_{AF}&0&\frac{1}{96}&0&0\\[0.5ex]
  S_{DA}&0&0&\frac{1}{1152}&0\\[0.5ex]
  S_{DD}&0&0&0&\frac{1}{24}\,.
  \end{array}
$$
\end{lm}
\begin{proof}
We first compute the classes of the four surfaces in homology. The classes of the curves $C_F$ and $C_A$ were given above, and we recall that the class of a point $[B]\in\oa[1]$ is equal to~$12L_1$, and has degree $1/2$ due to stackiness. We thus obtain
$$
 [S_{DD}]=12L_1D_2;\quad [S_{DA}]=12L_1(5L_2-D_2/2).
$$
To compute the classes of the other two surfaces, we note that these are simply pullbacks of classes from $\oa[2]$, except for the stacky issue: there is an extra factor  $1/2$ that comes from the involution $-\operatorname{id}$ on the first factor $\oa[1]$. We thus obtain that these classes are equal to
$$
 [S_{AF}]=6L_2D_2;\quad [S_{AA}]=12L_2(5L_2-D_2/2).
$$
At this point we note that the class of the surface $S_{AF}$ agrees with the formula
given in \cite[Sec. 2]{vdgeerchowa3}
for the class of the cycle $\beta_2$ on $\oa[2]$, which we denote by  $\beta_2^{(2)}$. This is indeed correct as $\beta_2^{(2)}$ is represented by the curve~$C_F$ understood as the fiber of the map $\partial \oa[2] \to \oa[1]$ over the cusp, which comes with a factor $1/2$ due to the involution $-\operatorname{id}$ on $\oa[1]$.

The intersection theory on~$\oa[2]$ is given in \cite[Sec.~2]{vdgeerchowa3}: the top intersection numbers are
$$
 D_2^3=-\frac{11}{12};\quad L_2D_2^2=-\frac{1}{24};\quad L_2^2D_2=0;\quad L_2^3=\frac{1}{2880}\,,
$$
from which we also obtain the intersection numbers
$$
M_2^3=\frac{1}{60}; \quad M_2L_2D_2=\frac{1}{24};\quad M_2^2D_2=\frac{1}{12}.
$$
Now computing the intersection numbers is a matter of substituting the formulas for the classes of the effective surfaces and using these intersection numbers.
\end{proof}
\begin{proof}[Proof of \Cref{prop:effoa1oa2}]
This follows immediately from the fact that all elements of our chosen basis $L_1L_2,L_1M_2,L_2^2,M_2^2$ of $H^4(\oa[2])$ are nef, being products of nef divisors, and the fact that the intersection matrix in the lemma above is diagonal.
Indeed, let $S\in\oE(\oa[1]\times\oa[2])$ be any pseudoeffective class. Since the classes $S_{DD},S_{DA},S_{AF},S_{AA}$ give a basis of $H_4(\oa[1]\times\oa[2])$, we can write the class $S$ uniquely as a linear combination $S=a_1S_{DD}+a_2S_{DA}+a_3S_{AF}+a_4S_{AA}$ with some coefficient $a_i\in\RR$. But since the class $L_1L_2$ is nef, we know that $S.L_1L_2\ge 0$. However, this intersection number is simply equal to $a_4/1152$, so we must have $a_4\ge 0$. Similarly intersecting with each of the other three nef generators of $H^4(\oa[1]\times\oa[2])$ gives non-negativity of another coefficient $a_i$. This shows the statement for $\oE(\oa[1]\times\oa[2])$. The claim for $\Nef(\oa[1]\times\oa[2])$ follows by duality.
\end{proof}

We now want to compute the classes of the surfaces on $\oa$ that are the images of our four surfaces on $\oa[1]\times\oa[2]$ under the morphism $N:\oa[1]\times\oa[2]\to\oa$. This if of course equivalent to computing the intersection numbers of these four surfaces $S_{AF},S_{AA},S_{DD},S_{DA}$ on the product $\oa[1] \times \oa[2]$ with the $N^*$-pullbacks of a basis of $H^4(\oa)$. What we shall compute here is the intersection of these surfaces with the $N^*$-pullbacks of
$L^2,M^2,LM$ and $\beta_2$ (where here the classes $L$ and $M$ stand for $L_3$ and $M_3$ on $\oa$). We shall see at the beginning of \Cref{sec:classes}
that these classes are indeed a basis of the $4$-dimensional space $H^4(\oa)$. This will be based on the computations of this section.

To compute the intersection numbers with these pullbacks  is straightforward using \Cref{lm:intersectionsoa1oa2} once we determine the pullbacks. These of course factorize, and are thus easily computed to be
\begin{equation}\label{eq:restrict}
\begin{aligned}
 N^*L=L_1+L_2;\qquad N^*M=M_1+M_2=M_2;\\
  N^*\beta_2=D_1D_2+\beta_2^{(2)}=(12L_1+6L_2)(12L_2-M_2),
\end{aligned}
\end{equation}
where $\beta_2^{(2)}$ denotes the class of the torus rank 2 locus on $\oa[2]$, which is, as we have already pointed out, equal to $6L_2D_2=6L_2(12L_2-M_2)$.

\begin{prop}\label{prop:intersectionsoa1oa2}
The intersection numbers of the generators  $S_{AF}$,$S_{AA}$,$S_{DD}$ and $S_{DA}$ of $\oE(\oa[1] \times \oa[2])$  with the $N^*$-pullback of the classes $L^2$,$LM$,$M^2$,$\beta_2\in H^4(\oa)$ are equal to
$$
  \begin{array}{r|cccc}
  &N^*(L^2)&N^*(LM)&N^*(M^2)&N^*(\beta_2)\\
  \hline\\[-2ex]
  S_{AF}&0&\frac{1}{96}&0&-\frac18\\[0.5ex]
  S_{DA}&\frac{1}{1152}&0&0&\frac{1}{16}\\[0.5ex]
  S_{DD}&0&\frac{1}{48}&\frac{1}{24}&-\frac18\\[0.5ex]
  S_{AA}&\frac{1}{576}&0&0&\frac18 \,.
  \end{array}
$$
\end{prop}
\begin{proof}
These computations are straightforward by combining \Cref{lm:intersectionsoa1oa2} with the formulas~\eqref{eq:restrict} for the pullbacks.
We shall give the details for the surface $S_{DD}$, the other surfaces being similar (and with all computations done by hand and rechecked by Maple to avoid arithmetic errors)
$$
\begin{aligned}
 S_{DD}.N^*(L^2)&=S_{DD}.(L_1+L_2)^2=S_{DD}.(2L_1L_2+L_2^2)=0;\\
 S_{DD}.N^*(LM)&=S_{DD}.(L_1+L_2)M_2=S_{DD}.(L_1M_2+L_2M_2)\\
                            &=S_{DD}.(L_1M_2+M_2^2/2)=0+\frac12\cdot\frac{1}{24}=\frac{1}{48};\\
 S_{DD}.N^*(M^2)&=S_{DD}.M_2^2=\frac{1}{24};\\
 S_{DD}.N^*(\beta_2)&=S_{DD}.(12L_1+6L_2)(12L_2-M_2)\\
                                   &=S_{DD}.(144L_1L_2+72L_2^2+12L_1M_2-6L_2M_2)\\
                                   &=-3S_{DD}.M_2^2=-\frac18\,,
\end{aligned}
$$
where we have used the fact that $L_1^2=0$, for dimension reasons, and the identity $M_2^2=2L_2M_2$.
\end{proof}

\begin{rem}\label{remSAAstackyfactor}
We note that $[S_{AA}]=2[S_{DA}]$. This seems surprising since both surfaces are mapped to the same surface $S_A$ in $\oa[3]$. The reason for this is the following: in the case of the surface $S_{DA}$ the
involution which interchanges the two factors of $\oa[1] \times \oa[1]$ is already taken care of in $\oa[2]$. In contrast to that the surface $S_{AA}$ maps $2$-to-$1$ onto the surface $S_A$ in $\oa[3]$, which means that
the intersection numbers on the pullback are multiplied by $2$.
\end{rem}

\section{The classes of effective surfaces}\label{sec:classes}
The cohomology of $\oa$ has been computed in~\cite{huto1}, and we recall that in particular $H^4(\oa)=\RR^4$.
We claim that a basis of cohomology can be given as
$$
 H^4(\oa)=\RR L^2\oplus \RR LM\oplus \RR M^2\oplus\RR\beta_2\,,
$$
where $L$ and $M=12L-D$ are the classes of the line bundles defined previously , and $\beta_2=\osi{1+1}$ denotes the fundamental class of the codimension two boundary stratum.
Indeed, the four classes $L^2,LM, M^2,\beta_2$ certainly give elements of $H^4(\oa)$. To see that they are linearly independent, one can trace through the spectral sequence computation of homology in~\cite{huto1}, but it will also follow from the non-degeneracy of the intersection matrix of these classes with our surfaces $S_A,S_F,S_D,\osi{K3+1},\osi{C4}$, which we will now compute. This is the main result of this section:
\begin{prop}\label{prop:intersections}
The intersection numbers of the classes in $H^4(\oa)$ and of our five effective surfaces are as follows
\begin{equation}\label{eq:intersect}
\begin{array}{r|cccr}
&L^2&LM&M^2&\beta_2\\
\hline&&&&\\[-2ex]
S_A&\frac{1}{1152}&0&0&\frac{1}{16}\\[0.5ex]
S_F&0&\frac{1}{96}&0&-\frac{1}{8}\\[0.5ex]
S_D&0&\frac{1}{48}&\frac{1}{24}&-\frac18\\[0.5ex]
\osi{C4}&0&0&\frac{1}{48}&-\frac{1}{16}\\[0.5ex]
\osi{K3+1}&0&0 &\frac14&\frac14.
\end{array}
\end{equation}
\end{prop}

One result of this computation is that the classes of these five surfaces generate a cone with five extremal rays. It is worth pointing out that these computations are very sensitive to even minor numerical inaccuracies.
\begin{cor}\label{cor:indep}
Among the classes of the surfaces $S_A,S_F,S_D,\osi{C4},\osi{K3+1}$ no one class lies in the cone generated by the other four.
\end{cor}
\begin{proof}
It suffices to show that the class of none of these five surfaces can be written as a linear combination of the other four classes with positive coefficients. While this computation is straightforward, numerically it depends crucially on the intersection numbers computed above.

Indeed, as $S_A$ is the only surface that has non-zero intersection number with $L^2$, it is linearly independent with the span of the remaining 4 surfaces, and cannot be involved in any linear relation with them. The remaining 4 classes span a 3-dimensional space, and thus satisfy precisely one linear relation. Since the only two classes that have non-trivial intersection with the class $LM$ are $S_D$ and $S_F$, this relation must either be that the classes~$\osi{C4}$ and~$\osi{K3+1}$ are proportional --- which is not the case --- or have the form $S_D-2S_F-a\osi{C4}-b\osi{K3+1}=0$. In the later case we compute $a=-1$ and $b=1/4$. Since the signs of $a$ and $b$ are opposite, there is thus no linear relation satisfied by $S_D$, $S_F$, $\osi{C4}$, $\osi{K3+1}$ such that one class is expressed as a non-negative linear combination of the other three (which would require one of the coefficients $1,-2,a,b$ to have a different sign than the other three).
\end{proof}

To prove the proposition we first observe that the claims about the surfaces $S_A$, $S_F$ and $S_D$ immediately follow from \Cref{prop:intersectionsoa1oa2} and \Cref{remSAAstackyfactor}, as these surfaces are images of the surfaces contained in~$\oa[1]\times\oa[2]$, and we computed the relevant intersection numbers on $\oa[1]\times\oa[2]$.
Hence it remains to deal with the classes of the boundary cones $\osi{C4}$ and $\osi{K3+1}$. The intersection numbers of these boundary strata with our basis of~$H^4(\oa)$ can be read off from the work of Tsushima~\cite[Prop.~3.4]{tsushima} or van der Geer~\cite{vdgeerchowa3}, if one is careful to match the stackiness convention and notation. We first deal with the surface $\osi{K3+1}$.

\begin{lm}\label{lm:classosiK3+1}
The surface $\osi{K3+1}$ has the following intersection numbers with a basis of $H^4(\oa)$:
$$
 \ L^2.\osi{K3+1}=LD.\osi{K3+1}=LM.\osi{K3+1}=0 \, ,
$$
$$
D^2.\osi{K3+1}=M^2.\osi{K3+1} = \beta_2.\osi{K3+1}=\frac14\,  .
$$
\end{lm}
\begin{proof}
Indeed, van der Geer~\cite[Section 3]{vdgeerchowa3} computes the following intersection numbers on $\oa$:
$$
 LD^2\sigma_3=\frac{1}{48};\quad L\beta_2\sigma_3=\frac{1}{48}\,.
$$
Here in the notation of~\cite{vdgeerchowa3} the class $\sigma_1$ is equal to our $D_1$. We recall that  $\beta_3$ is the preimage of~$\ab[0]$ under the map $\pi:\oa \to \Sat[3]$.
To explain the cycle~$\sigma_3$ it is convenient  to work on the level $2$ cover~$\oa{}(2) $ of~$\oa$. We recall that the (reduced) preimage of the irreducible boundary divisor $D\subset\oa$ on the cover $\oa{}(2)$
is the union $\cup D_i$ of $2^{2g}-1$ irreducible components indexed by points of $(\ZZ/2\ZZ)^{2g}\setminus\lbrace 0\rbrace$.
We also recall the following basic fact about non-empty triple intersections of the form $D_i \cap D_j \cap D_k$. Such an intersection is mapped under the map $\pi:\oa(2) \to \Sat[3](2)$ to either  ${\Sat[2]}(2)$ or to ${\Sat[1]}(2)$, and we call such an
intersection {\em local} or {\em global} respectively, see ~\cite[Section 5]{ergrhu2}.
These two cases correspond to the
cones $\si{K3}=\langle x_1^2,x_2^2,(x_1-x_2)^2\rangle$  and $\si{1+1+1}= \langle x_1^2,x_2^2, x_3^2\rangle$ respectively. In the first case we will write these intersections as $D_i \cap D_j \cap D_{i+j}$, whereas in the second case
we use the notation $D_i \cap D_j \cap D_k$. The cycle $\sigma_3$ is defined as the sum of the images of all triple intersections, including both the local and global cases. This shows that  $\sigma_3=\osi{K3}+\beta_3$.

Since $L$ is a pullback from $\Sat$, and $\beta_3$ is mapped to the point $\oa[0]\in\Sat$, it follows that $L|_{\beta_3}$ is trivial.
We further recall from \Cref{sec:stratification}, and in particular Equation \eqref{eq:strataclosures}, that $\osi{K3}= \si{K3} \cup \osi{K3+1}$. This shows that the boundary of $\osi{K3}$ is the preimage of $\{i\infty\}$ under the map to $\oa[1]$. Its
class is given by $12L$ and we obtain the relation
$$
12L\osi{K3}=\osi{K3+1}.
$$
(For a detailed discussion of the geometry of $\osi{K3}$ we also refer to \Cref{sec:osiK3}.)
Using this relation we compute
$$
 12LD^2\sigma_3=D^2\osi{K3+1}=\frac{12}{48}=\frac14, \, 12L\beta_2\sigma_3=\beta_2\osi{K3+1}=\frac{12}{48}=\frac14\,.
$$
\end{proof}
\begin{rem}\label{rem:reprooosiK3+1}
We shall give a different computation for the class of $\osi{K3+1}$ later after \Cref{rem:permutations}, thus providing a further check for our computations above .
\end{rem}
To compute the intersection numbers of $\osi{C4}$ with a basis of $H^4(\oa)$, we further use van der Geer's computations from~\cite[Section 3]{vdgeerchowa3}: he computes
$$
D^2\sigma_4=\frac{13}{48};\quad \beta_2\sigma_4=\frac{3}{16}
$$
where $\sigma_4$ is the pushforward of the  intersection $\sum_{i<j<k<\ell}D_iD_jD_kD_{\ell}$ on $\oa{}(2)$ to $\oa$.
Thus $\sigma_4=\osi{K3+1}+\osi{C4}$ in our notation, since these are the only codimension 4 cones.
\begin{lm}
The surface $\osi{C4}$ has the following intersection numbers with a basis of $H^4(\oa)$:
$$
 L^2.\osi{C4}=LD.\osi{C4}= ML.\osi{C4}=0, \,  D^2.\osi{C4}=M^2.\osi{C4}=\frac{1}{48}, \, \beta_2.\osi{C4}=-\frac1{16}\,.
$$
\end{lm}
\begin{proof}
Since $\osi{C4}\subset\beta_3$ and $L|_{\beta_3}$ is trivial, it follows that $L|_{\osi{C4}}$ is zero.

We now compute the other two intersection numbers:
$$
\begin{aligned}
 D^2.\osi{C4}&=D^2.\sigma_4 -D^2.\osi{K3+1}=\frac{13}{48}-\frac14=\frac{1}{48},\\
 \beta_2.\osi{C4}&=\beta_2.\sigma_4- \beta_2\osi{K3+1}=\frac{3}{16}-\frac14=-\frac{1}{16}\,.
\end{aligned}
$$
\end{proof}

\section{The intersection theory on $V=\ox[1]\times\ox[1]/S_2$}\label{sec:ox1ox1}
The crucial special feature of genus 3 is the fact that the boundary components of a level cover~$\oa(n)$ can be interpreted as universal abelian surfaces.
As this will play a major role for us and since the relevant geometry turns out to be surprisingly delicate and is not much discussed in the literature, we will briefly describe this here.

Let $D_i(n)$ be a boundary component of $\oa(n)$.Then under the contraction morphism $\pi:\oa(n)\to\Sat(n)$ the component $D_i(n)$ is mapped to a copy of $\Sat[2](n)$. As discussed in~\cite[Section 4]{tsushima} and~\cite[Section 3]{huleknef}, it turns out that the restriction $\pi|_{D_i(n)}: D_i(n) \to \Sat[2](n)$ factors through a morphism $f_n: D_i(n) \to \oa[2](n)$. Furthermore,~$f_n$ extends the universal family $\ua[2](n) \to \ab[2](n)$ in such a way that the fiber over a boundary point of $\ab[2](n)$ is isomorphic to the $2n$ times principally polarized
semi-abelic surfaces (with level~$n$ structure) parameterized by this point.
We shall therefore write this as $p_n: \ox[2](n) \to \oa[2](n)$.
The map~$f_n$ is equivariant with respect to the stabilizer of the given cusp within the deck transformation group of the cover $\oa[2](n)\to\oa[2]$. Taking the quotient, we obtain a
fibration $p: \ox[2] \to \oa[2]$ which extends the universal Kummer family over $\ab[2]$.

On the other hand, we also have a finite map $f: \ox[2] \to \partial \oa[3]$. It turns out that it is not an isomorphism, as elements of the deck transformation group of the cover $\oa(n)$ which do not stabilize the chosen boundary divisor can and do introduce further identifications. For this we recall
that both $\ox[2](n)$ and $\oa[2](n)$ are toroidal varieties and that the map between them is a morphism of toroidal varieties.  If we work with the standard cusp corresponding to $x_1^2$ (see Table~\ref{table:cones}), the relevant lattices are $\Sym^2(\ZZ^3)/\langle x_1^2\rangle$
and $\Sym^2(\ZZ^2)$ respectively, where we use $x_2,x_3$ for the target lattice, and the map is given by setting $x_1=0$. Now consider the cones $\langle x_1^2, x_2^2, x_3^2, (x_2-x_3)^2\rangle$ and
$\langle x_1^2, x_2^2, (x_1-x_2)^2,x_3^2 \rangle$, which both give rise to $2$-dimensional strata in the chosen boundary component of $\oa(n)$. Both of these cones
lie in the same $\GL(3,\ZZ)$-orbit as $\si{K3}$ and hence these strata will
be identified in $\partial  \oa$. However, they give two different surfaces in $\ox[2](n)$ which behave differently under the map  $p_n: \ox[2](n) \to \oa[2](n)$ and will therefore not be identified under
the stabilizer group of the chosen cusp. In the first case the cone is
mapped to  the $3$-dimensional cone $\langle x_2^2, x_3^2, (x_2-x_3)^2\rangle$, and thus the corresponding toroidal stratum is mapped to a deepest point in the boundary of $\oa[2](n)$, i.e. a point where two components of an $n$-gon over a cusp of $\oa[1](n)$ meet.
In the second case the cone is mapped to the $2$-dimensional cone $\langle x_2^2,x_3^2 \rangle$, and this stratum is mapped  to a component of an $n$-gon.

\smallskip
For our purposes it will be enough to use the finite map $f: \ox[2] \to \partial \oa$ and we will use this to prove that  the surfaces $S_D,S_F,S_A$ are extremal effective in $\Eff(\partial\oa)$,
by arguing the extremality of their preimages on $\ox$ under $f$. Then since the map $\pi$ contracts $S_D$ and $S_F$ to curves, we will be able to deduce that they are extremal effective on $\oa$, thus finally proving extremality of $S_D$, and reproving the extremality of~$S_F$.

It will turn out that the preimage in $\ox$ of the decomposable locus $\oa[2]$ plays an especially important role in this geometry.
As usual, we will think of the decomposable locus as $\oa[1] \times \oa[1]$ modulo the involution interchanging the two factors.
The inverse image under the map  $f: \ox \to \oa[2]$ of the decomposable locus is isomorphic to $V:=\ox[1]\times\ox[1]/S_2$, where the symmetric group $S_2$ permutes the two factors.
We will study $\oE(V)$ using the same approach as for $\oE(\oa[1]\times\oa[2])$.
For this we recall that $H^2(\ox[1])=\RR L\oplus \RR T$, with the intersection theory governed by \Cref{prop:intersectionX1}. To find generators for $H^4(V)$, we note that $$H^4(\ox[1]\times\ox[1])=H^0(\ox[1])\otimes H^4(\ox[1])\oplus H^2(\ox[1])\otimes H^2(\ox[1])\oplus H^0(\ox[1])\otimes H^4(\ox[1]),$$ and $H^4(V)$ is then simply the subspace of $H^4(\ox[1]\times\ox[1])$ invariant under the involution interchanging the two factors. We thus obtain
$$
 H^4(V)=\RR L_1L_2\oplus \RR T_1T_2\oplus \RR(L_1T_1+L_2T_2)\oplus\RR(L_1T_2+L_2T_1),
$$
where $L_i$ and $T_i$ denote the pullbacks of the classes $L$ and $T$ from the corresponding factor.
\begin{prop}\label{prop:effV}
The cone of nef codimension two classes on $V$ is generated by the four generators identified above:
$$
 \Nef(V)=\RR_{\ge 0} L_1L_2+ \RR_{\ge 0} T_1T_2+ \RR_{\ge 0}(L_1T_1+L_2T_2)+\RR_{\ge 0}(L_1T_2+L_2T_1),
$$
while the cones of pseudoeffective and effective surfaces on $V$ coincide, and are equal to
$$
 \oE(V)=\RR_{\ge 0} L_1L_2+ \RR_{\ge 0} Z_1Z_2+ \RR_{\ge 0}(L_1Z_1+L_2Z_2)+\RR_{\ge 0}(L_1Z_2+L_2Z_1).
$$
\end{prop}
Here we recall that $T=Z+L$ on $\ox[1]$, where $Z$ is the $0$-section $\oa[1]\to\ox[1]$, and we denoted $Z_i$ the pullbacks from the two factors. Observe in particular that $L_1Z_1+L_2Z_2=L_1T_1+L_2T_2$, since $T_i=Z_i+L_i$ and $L_i^2=0$.
\begin{proof}
To start with, we observe that $\Eff[1](\ox[1])=\RR_{\ge 0}L+\RR_{\ge 0} Z$, while the nef cone is $\Nef[1](\ox[1])=\RR_{\ge 0}L+\RR_{\ge 0} T$. Indeed, by definition, the classes $L$ and $T$ are nef, and thus to prove that they are the extremal rays of
the (in this case two-dimensional)  effective cone it suffices to show that they have zero intersection numbers with some effective classes, which are of course simply the classes~$L$ and~$Z$ respectively, by \Cref{prop:intersectionX1}.
We can now compute the intersection numbers on $V$ using \Cref{prop:intersectionX1} and taking the involution by $S_2$ into account. This is straightforward and results in the following matrix of intersection numbers:
$$
  \begin{array}{r|cccc}
  &L_1L_2&T_1T_2&L_1T_1+L_2T_2&L_1T_2+L_2T_1\\
  \hline\\[-2ex]
  L_1L_2&0&\frac{1}{1152}&0&0\\[0.5ex]
  Z_1Z_2&\frac{1}{1152}&0&0&0\\[0.5ex]
  L_1Z_1+L_2Z_2&0&0&\frac{1}{576}&0\\[0.5ex]
  L_1Z_2+L_2Z_1&0&0&0&\frac{1}{576}\,.
  \end{array}
$$
In particular, we observe that each of the claimed generators of $\Nef(V)$ has non-zero pairing with exactly one of the claimed generators of $\oE(V)$, and that this pairing is positive.
Since  the 4 claimed generators of the nef cone are clearly nef classes, and the 4 claimed generators of the pseudoeffective cone are effective classes, the same argument as in the proof of \Cref{prop:effoa1oa2} applies verbatim.
\end{proof}

We now compute the classes of the images in $\oa$ under the map $f$ of the 4 effective surfaces above that generate the cone $\oE(V)$. For this, we compute their intersection numbers with the pullback of the basis of $H^4(\oa)$ pulled back to $H^4(\partial\oa)$ under the inclusion $\partial\oa\hookrightarrow\oa$, and then pulled back further under $f$.
Thus we need to compute the restrictions of the $f$-pullbacks of classes $L,M$ and $\beta_2$ to~$V$.
In order to avoid unnecessarily complicated notation we shall denote the pulled back classes to~$V$ by the same letter as the classes in $\oa$, dropping $f^*$ in notation. We clearly have $L|_V=L_1+L_2$. To understand the restriction of the boundary and of $\beta_2$ to $V$ is more complicated and requires going to a level cover, where we can use the fibration  $p_n: \ox[2](n) \to \oa[2](n)$. On a level cover, every component~$V_i(n)$ of~$V(n)$ is a product of two universal elliptic curves $\ox[1](n) \times \ox[1](n)$. Note that here we mean the direct product, not the fiber product over $\oa[1](n)$, which will be considered in \Cref{sec:osiK3}. We also recall that the fiber of the universal elliptic curve $\ox[1](n) \to \oa[1](n)$ over a cusp is an $n$-gon of rational curves. Fixing a cusp in each of the two~$\oa[1](n)$, factors we label the corresponding $n$-gons by $\ell^k_j$ where $j\in \{0, \dots , n-1\}$ is to be understood cyclically and $k\in \{1,2\}$. There is a unique component $D_1$ of the boundary $\partial\oa{}(n)$ that contains $V_i(n)$.  The class of $D_1$ restricted to itself is minus double the universal theta divisor trivialized along the zero section, and thus $D_1|_{V_i(n)}=-2(T_1+T_2)$ where we also use $T_i$ for the class
on the level cover.  Let $D_2$ be another boundary component that intersects $V_i(n)$. Then this intersection $D_2 \cap V_i(n)$  is of the form $\ox[1](n) \times \ell_j^2$ or $\ell_m^1  \times \ox[1](n)$. Adding these up, we obtain the classes of the products
with the degenerate fibers, which is just $12L_1 + 12L_2$. Hence finally
\begin{equation}\label{equ:DrestV}
D|_V= -2(T_1+T_2) + 12(L_1 + L_2)
\end{equation}
and therefore
\begin{equation}\label{equ:MrestV}
M|_V= 2(T_1+T_2).
\end{equation}

In order to compute $\beta_2|_V$ we have to understand the pushforward of the intersection $\sum_{i<j}D_iD_j$ on a level cover to $V$.
There are three possible combinations of pairs $D_i, D_j$ which we have to take into account. The first is that one of the $D_i$ is the component $D_1$ containing $V_i(n)$. Fixing this and
summing over all other components gives a contribution $-2(T_1+T_2)\cdot 12(L_1+L_2)$, as computed above. The second possibility is that, say $D_2$, intersects $V_i(n)$ in $\ox[1](n) \times \ell_j^2$ and $D_3$ intersects $V_i(n)$ in $\ell_m^1  \times \ox[1](n)$.
Then the intersection of $D_2$ and $D_3$ is the surface $\ell_m^1 \times \ell_j^2$, and summing up over all possible such pairs of irreducible components of the boundary gives the product of the singular fibers, which is the class  $12L_1\cdot 12L_2$. Finally we could have that
$D_2$ intersects $V_i(n)$ in  $\ell_m^1  \times \ox[1](n)$ and that $D_3$ intersects $V_i(n)$ in  $\ell_k^1  \times \ox[1](n)$ such that $\ell_m^1$ and $\ell_k^1$ intersect on $V_i(n)$ in a point $P$. Then we obtain that geometrically
$D_2\cap D_3$ is $\{P\} \times \ox[1](n)$. Summing up over all such pairs of boundary components gives the class $12(L_1Z_1+L_2Z_2)$. Adding up the three possible types of intersections, we finally obtain
 \begin{equation}\label{equ:beta2restV}
\beta_2|_V=-24(T_1+T_2)(L_1+L_2) + 144L_1L_2 +  12(L_1Z_1+L_2Z_2).
\end{equation}

Denoting the generators of $\oE(V)$ by $S_1:=144L_1L_2$ (the class of a fiber over the most degenerate points), $S_2:= Z_1Z_2$ (the intersection of the two zero sections), $S_3:=12(L_1Z_1+L_2Z_2)$ (the class of a point on one factor times the other factor, and vice versa), and $S_4:=12(L_1Z_2+L_2Z_1)$ (the class of a fiber of the map to $\oa[1]$ on one factor, times the zero section on the other, and vice versa), we then compute the intersection numbers (all of them computed using Maple for safety):
\begin{prop}\label{prop:intersectionsV}
The intersection numbers of the generators $S_1,S_2,S_3,S_4$ of $\oE(V)$ with the pullback of a basis of $H^4(\oa)$ under $f$ are equal to
$$
  \begin{array}{r|cccc}
  &L^2&LM&M^2&\beta_2\\
  \hline\\[-2ex]
  S_1&0&0&1&0\\[0.5ex]
  S_2&\frac{1}{576}&0&0&\frac18\\[0.5ex]
  S_3&0&\frac{1}{24}&\frac{1}{12}&-\frac14\\[0.5ex]
  S_4&0&\frac{1}{24}&0&-\frac12\,.
  \end{array}
$$
\end{prop}
By inspection, we see that geometrically the images of these surfaces are
$$
  f(S_2)=S_A,\quad f(S_3)=S_D,\quad\mathrm{ and } \quad f(S_4)=S_F\,
$$
where by these equalities we mean reduced images, i.e.~the classes will be proportional up to an overall common factor. To understand these factors we first remark that the surfaces $S_i$ are all contained in the fiber over a torus rank 1 cusp.
The above computations take the involution $- \operatorname{id}_{\Sp(4,\ZZ)}$ into account where $\Sp(4,\ZZ)$ is contained in the stabilizer of this cusp. However, there is another involution  $\iota$ in $\Sp(4,\ZZ)$ such that the
product of $\iota$ and $- \operatorname{id}_{\Sp(4,\ZZ)}$ is equal to $- \operatorname{id}_{\Sp(6,\ZZ)}$ (see also the discussion in \Cref{rem:stacky12}).
Hence, in order to interpret the above numbers in $\oa$, they all have to be divided by a factor 2. The numbers thus obtained for the surfaces $S_A$ and $S_D$ that are the images of $S_2$ and $S_3$ then agree with the intersections computed in \Cref{prop:intersections}.  The reason why there is a further factor $1/2$ for the class of $f(S_4)$ compared to the class of $S_F$ computed  in \Cref{prop:intersections} lies in our definition of the surface $S_F$ which was defined as the image of  the irreducible
surface $\oa[1] \times C_F$ in $\oa[1] \times \ab[2]$ under the map   $N:\oa[1] \times \oa[2] \to \oa,$ whereas $S_4$ is represented by a reducible surface on $V$, with two components, each of which is mapped onto $S_D$.
Finally, the surface $S_1$ is the class of a generic fiber of the map $\ox\to\oa[2]$, and from \Cref{prop:intersections} it follows that $2[f(S_1)]=12[\osi{C4}]+3[\osi{K3+1}]$. While the surface $S_1$ is thus clearly not extremal effective on $\oa$, in the next section we will be able to use the above to prove extremality of surfaces $S_D$ and $S_F$ on $\oa$.

\section{The geometry of the stratum $\osi{K3}$}\label{sec:osiK3}
In this section we will discuss the rich geometry of the closed stratum $\osi{K3}\subset\oa$. There are three reasons for doing this.
Firstly, this is a threefold of independent interest, with rewarding and intricate geometry. Secondly, we need some of these results to prove in \Cref{sec:extremality} that the boundary strata $\osi{K3}$ and $\osi{C4}$ are extremal effective, and thirdly this allows us to recompute the homology classes of the surfaces $S_D$ and $\osi{K3+1}$, which can be represented by effective surfaces contained in $\osi{K3}$, thus reproving parts of \Cref{prop:intersections}. Indeed, the stratum $\osi{K3+1}$ is contained in $\osi{K3}$ by definition, while after discussing the geometry of $\osi{K3}$ we will be able to see that the surface $S_D$ is homologous to a surface contained in $\osi{K3}$.

Consider the fiber product   $\oub(n):=\ox[1](n) \times_{\oa[1](n)}\ox[1](n)$, together with the action of the group
$H(n):=G(n) \times_{\SL(2,\ZZ/n\ZZ)} G(n)$, which  makes sense as the group $G(n)$ has a natural projection onto $\SL(2,\ZZ/n\ZZ)$. Note that we can compute
\begin{equation}\label{equ:orderHn}
|H(n)|=2n^5\,t(n)\,.
\end{equation}
To simplify notation we write $\oub := \oub(1)$.

We also note that $H(n)$ contains the element $(-1,-1)$ which acts as the Kummer involution on the fibers, but it does not contain the involutions which act as Kummer involution on one factor of the fiber, but not on the other. The reason we consider $H(n)$, rather than the full symmetry group, is the relationship between $\oub(n)$ and the toroidal compactification $\oa[3]$. As we shall see, a small blow-up of  $\oub(n)$, modulo the action of $H(n)$, is mapped to the stratum $\sigma_{K3}$. The Kummer involution then corresponds to the symmetry $(x_1,x_2) \mapsto (-x_1,-x_2)$ of $\sigma_{K3}$, whereas the Kummer involution on, say the first factor only, corresponds to $(x_1,x_2) \mapsto (-x_1,x_2)$  which is not an element in the automorphism group of $\sigma_{K3}$.

The fiber product $\oub(n)$ has $A_1$-singularities at all the points of the form $(P_1,P_2)$ where $P_1$ and $P_2$ are singular points on the fiber over a cusp. This follows by an easy local computation. Indeed, the projections on each of the fibrations $\ox[1](n) \to \oa[1](n)$ are locally, near a singular point of an $n$-gon, given by $(x,y) \mapsto xy=t$ and
$(u,v) \mapsto uv=t$ respectively, such that the fiber product is given by $xy-uv=0$.
There are $n^2$ nodes over each cusp of $\oa[1](n)$ and hence $n^2\,t(n)$ nodes in total. The group~$H(n)$ acts transitively
on these  with stabilizers of order $2n^3$. The stabilizer group contains a conjugate of the Kummer involution and three cyclic groups $\ZZ/n\ZZ$. Two of these come from multiplication by $n$-th roots of unity on each of the components of the
$n$-gons meeting at these points, and the third is the stabilizer of the cusp in the base. We also note that the involution interchanges the two components of the $n$-gons meeting at such a point simultaneously and hence preserves the rulings of the
quadric $xy-uv=0$.

For each of the nodes on  $\oub(n)$ we can define two small resolutions.
Recall that these are given by blowing up a smooth surface going through such a node.
The two possible small resolutions
correspond to the two families of planes in a quadric cone given by the rulings of the quadric $xy-uv=0$. The total spaces thus obtained are related by an Atiyah flop.
By the above discussion we can choose one node and one of the two rulings associated to this node and then use the
group~$H(n)$ to define a small resolution for all nodes. Clearly, the group~$H(n)$ acts on this small resolution. This process gives us two possible small resolutions of $\oub(n)$.
These are isomorphic as abstract varieties, but not as resolutions. Indeed, if we add the involution which acts as the Kummer involution on one of the factors, but not on the other, it is easy to construct
an involution which fixes a node $(P_1,P_2)$ but interchanges the two rulings. We note, however, that the involution which we have just considered is not contained in $H(n)$.

To choose a small resolution we consider the diagonal $\Delta \subset  \oub$. Then $\Delta \cong \ox[1]$  is smooth and invariant under the Kummer involution acting diagonally on both factors, where both statements should be interpreted in the stack sense.
Hence, in the light of the above discussion,
this defines a small resolution of $\wub(n)$ of $\oub(n)$.  We shall here also use the shorter notation  $\wub:= \wub(1)$. We note that
the involutions which are the Kummer involution on one of the two factors, but not on the other,
do not act on $\wub(n)$. This fits in with the fact that
such an involution does not preserve the diagonal, but maps it to the antidiagonal.  We also note that the strict transform $\wDelta$ in $\wub$ is again smooth and that it is the blow-up of $\Delta$ in the node, again in the stack sense.
We denote the inverse image of $\Delta$ in $\oub(n)$ by $\Delta_n$ and its inverse image in $\wub(n)$ by  $\wDelta_n$ respectively. We note that $\Delta_n$ and  $\wDelta_n$ both consist of
$n^2$ irreducible components and that each component of $\wDelta_n$ is the blow-up of the corresponding component of $\Delta_n$ in the nodes of $\oub(n)$ through which this component passes. Every component of $\Delta_n$ contains~$n\,t(n)$ such nodes and through any such node $n$ components of  $\Delta_n$ pass. These statements follow from an easy counting argument using that the stabilizer of a component of $\Delta_n$ in $H(n)$ is isomorphic to~$G(n)$ and the description of the stabilizers of the nodes of $\oub(n)$ which we gave earlier.

Finally we remark that the small resolution $\wub(n)$ is projective. This follows for example from \Cref{lm:osiK3geometry}  which shows that there is a finite map from $\oub(n)$ to the stratum $\si{K3}$ contained in the projective variety $\oa[3]$.

We shall now define various line bundles on  $\oub(n)$  and $\wub(n)$ respectively, whose intersection numbers we will compute.
Since $\oub(n)$ is a small resolution of  $\wub$ we can identify their Picard groups and we will, in order to simplify notation, use the same symbol for the corresponding line bundles on both spaces.
As usual, these are $\QQ$-line bundles which pull back to honest line bundles on $\oub(n)$ for
$n \geq 3$. If $N$ is such a line bundle, we will denote its pullback to the level~$n$ cover by $N_n$. As before, we will compute the intersection numbers by computing them on a level cover $\wub(n)$ and then dividing by the order of the
group $H(n)$.

As before we denote by $L$ the Hodge line bundle on $\oub$, which is a pullback from the base $\oa[1]$. Further, let $Z_i$, for $i=1,2$ be the pullbacks of the $0$-sections on the two factors of $\oub$, and denote by~$T_i, i=1,2$
the pullbacks of the theta divisors trivialized along the $0$-sections from the two factors.

\begin{prop}\label{prop:rankPictildeY}
The Picard group  $\Pic_\RR(\wub)$ has rank $4$ and is generated by $L,Z_1,Z_2$ and $\wDelta$. Moreover, $\Pic_\RR(\wub)\cong H^2(\wub,\RR)$.
\end{prop}
\begin{proof}
The general fiber of $\wub \to \oa[1]$ is a product $E \times E$ where $E$ is a general elliptic curve. The Picard group of $E \times E$ has rank $3$ (and is
generated by the factors and the diagonal). These divisors are restrictions of globally defined divisors, namely $Z_1, Z_2$ and $\wDelta$ respectively. Any irreducible effective divisor independent of these must be supported on a fiber. Since all fibers,
including the  fiber over the cusp, are irreducible and a multiple of $L$ the claim about the Picard group follows.

To prove the second statement it suffices to argue that for a smooth level cover the invariant part of $H^{2,0}(\wub(n))$ vanishes. Indeed, using the fact that $\ox[3] \to \oa[3]$ is the Hesse pencil, i.e. the projective plane blown up in $9$ points,
it is easy to check that $H^2(\wub(3), {\calO}_{\wub(3)})=0$, cf. \cite[Proof of Prop.~(7.1)]{schoen}.
\end{proof}

We have now defined all the relevant divisors and are ready to start the computation of the intersection numbers.
\begin{prop}\label{prop:tildeY}
The triple intersection numbers on $\wub$  are as follows:
$$
 \begin{array}{r|rrrrrr}
 &L&Z_1&T_1&Z_2&T_2&\wDelta\\
 \hline\\[-2ex]
 L.T_1=T_1^2=L.Z_1=-Z_1^2&0\ &0\ &0\ &\ \frac{1}{24}&\ \frac{1}{24}&\ \frac{1}{24}\\[0.5ex]
 L.T_2=T_2^2=L.Z_2=-Z_2^2&0\ &\ \frac{1}{24}&\ \frac{1}{24}&0\ &0\ &\ \frac{1}{24}\\[0.5ex]
 L.\wDelta&0\ &\ \frac{1}{24}&\ \frac{1}{24}&\ \frac{1}{24}&\ \frac{1}{24}&0\ \\[0.5ex]
 Z_1.Z_2=Z_1.\wDelta=Z_2.\wDelta&\ \frac{1}{24}&-\frac{1}{24}&0\ &-\frac{1}{24}&0\ &-\frac{1}{24}\\[0.5ex]
 \wDelta.\wDelta&0\ &-\frac{1}{24}&-\frac{1}{24}&-\frac{1}{24}&\ -\frac{1}{24}&-\frac{1}{2}
 \end{array}
$$
Here the equalities in the left-hand column are to be read as equalities of cycles in homology, and are part of the statement of the theorem.
\end{prop}

We will split up the  proof of this proposition into several  lemmas. We also remark that the two top rows in the left-hand-column are pullbacks of cycles from a factor  $\ox[1]$, in other words equalities of  intersection numbers on this surface.

\begin{lm}
We have
\begin{equation*}
LT_1^2=LT_2^2=LZ_1^2=LZ_2^2=0\,.
\end{equation*}
\end{lm}
\begin{proof}
On $\oub$ this follows by dimension reasons since all of these intersections are pullbacks of an intersection of three divisors on a surface. Note, moreover, that the class $12L$ on $\overline{\calA_1}$ is the class of a point, and thus $12L$ on $\oub$ can be represented as the fiber over a generic point of $\calA_1$. Thus the above intersections are supported away from the exceptional locus of the small resolution and hence the intersection numbers remain unchanged when pulled back to $\wub$.
\end{proof}

\begin{lm}\label{lem:intersection1}
The following holds:
\begin{equation*}
LT_1T_2=LZ_1Z_2=LT_1Z_2=LZ_1T_2=\frac{1}{24}\,.
\end{equation*}
\end{lm}
\begin{proof}
By \Cref{prop:intersectionX1}  we know that $T_i=Z_i+L$, for $i=1,2$. Together with $L^2=0$ this shows that we only have to calculate one of these intersection numbers. We shall compute $LZ_1Z_2$, as this intersection
is easy to understand geometrically. First, we use that $F=12L$. We now work on a level cover and consider a general fiber of $\ox[1](n) \times_{\oa[1](n)}\ox[1](n)\to\oa[1](n)$, which is of the form
$E \times E$ where $E$ is an elliptic curve. The surface $Z_1(n)$ intersects this in $E[n] \times E$ where $E[n]$ denotes the group of $n$-torsion points on $E$. Similarly  $Z_2(n)$ intersects this in $E \times E[n]$.
Hence $Z_1(n)$ and $Z_2(n)$ intersect on a general fiber in $n^4$ points. Using~\eqref{equ:orderHn} this gives
\begin{equation*}
Z_1Z_2L=\frac{1}{12}Z_1Z_2F= \frac{1}{12\,|H(n)|}\cdot n^4 \cdot n\,t(n)= \frac{1}{24}\,.
\end{equation*}
\end{proof}

In order to simplify some arguments later we note the following equalities
\begin{lm}\label{lem:cycles1}
The following equalities of cycles hold in homology:
\begin{equation*}
LT_i=T_i^2=LZ_i=-Z_i^2\,.
\end{equation*}
\end{lm}
\begin{proof}
This follows from $Z_i=T_i-L$ and $L^2=0$.
\end{proof}

\begin{lm}\label{lem:equintersection2}
We have the following intersection numbers
\begin{equation*}
LT_1\wDelta=LT_2\wDelta=LZ_1\wDelta=LZ_2\wDelta=\frac{1}{24}\,.
\end{equation*}
\end{lm}
\begin{proof}
Again, it is sufficient to prove one of these equalities. Indeed, the intersection of say $Z_1$ and $\wDelta$ on a general fiber $E \times E$ consists of all $n$-torsion points of this fiber, of which there are $n^4$. The rest of the proof
is analogous to the proof of \Cref{lem:intersection1}.
\end{proof}

\begin{lm}
The following holds:
\begin{equation*}
T_1^2\wDelta=T_2^2\wDelta=-Z_1^2\wDelta=-Z_2^2\wDelta=\frac{1}{24}\,.
\end{equation*}
\end{lm}
\begin{proof}
This is an immediate consequence of \Cref{lem:cycles1} and \Cref{lem:equintersection2}.
\end{proof}

\begin{lm}
The following holds:
\begin{equation*}
\wDelta^2L=0\,.
\end{equation*}
\end{lm}
\begin{proof}
Since $12L=F$, this means geometrically we can again represent $L$ by a multiple of a smooth fiber $E\times E$, for a smooth elliptic curve~$E$. But then the self-intersection of the diagonal in $E \times E$ is equal to $0$.
\end{proof}

\begin{lm}\label{equ:cycles2}
The following equality of cycles holds in homology:
\begin{equation*}
Z_1Z_2=Z_1\wDelta=Z_2\wDelta\,.
\end{equation*}
\end{lm}
\begin{proof}
This is geometrically clear as these cycles restrict to pairs of $n$-torsion points on all fibers (including the singular ones), see our discussion on the geometry of the universal elliptic curve.
We also note that the $n$-torsion points define sections of the universal elliptic curve $\ox[1](n)\to\oa[1](n)$. In particular these sections do
not meet singular pints of fibers and hence are not affected by the small resolution.
\end{proof}

It is easy to check that we have now computed all intersection numbers in \Cref{prop:tildeY} with the exception of $\wDelta^3$,  which we will now compute.

\begin{lm}\label{lm:Delta3}
We have
\begin{equation*}
\wDelta^3= - \frac{1}{2}\,.
\end{equation*}
\end{lm}
\begin{proof}
We first recall that $\wDelta_n$ consists of $n^2$ irreducible components which we denote by $\wDelta_{n,i}$. We first understand the pairwise intersections $\wDelta_{n,i}\cap \wDelta_{n,j}$ for $i\neq j$ . First of all, on $\oub(n)$ any two different components ${\Delta}_{n,i}$ and ${\Delta}_{n,j}$ intersect in (some of) the $n^2\,t(n)$ singular points of  $\oub(n)$, but nowhere else. For any such singular point of $\oub(n)$, there are $n$ components ${\Delta}_{n,j}$ that contain it. We also recall that in terms of the tangent cone at such a singularity, ${\Delta}_{n,i}$ and ${\Delta}_{n,j}$ correspond to planes in the same family and that the small blow-up is achieved by blowing up in these Weil divisors. This means that  $\wDelta_{n,i}$ and $\wDelta_{n,j}$ are obtained by blowing up each of ${\Delta}_{n,i}$ and ${\Delta}_{n,j}$ in a smooth point. Moreover,~$\wDelta_{n,i}$ and~$\wDelta_{n,j}$ intersect along a smooth rational curve, which is a
$(-1)$ curve on each of these surfaces. Hence we find that
\begin{equation}
\wDelta_{n,i}|_{\wDelta_{n,j}} = \sum_k E_k
\end{equation}
where $E_k$ are $(-1)$ curves on $\wDelta_{n,i}$ and $k$ runs through all the points of intersection of  ${\Delta}_{n,i}$ and ${\Delta}_{n,j}$.

We also claim that
\begin{equation}\label{equ:selfintdiag}
\wDelta_{n,i}|_{\wDelta_{n,i}} = -L + \sum_{\ell} E_{\ell}
\end{equation}
where the $E_{\ell}$ are $(-1)$ curves as before and $\ell$ runs through all singular points of  $\oub(n)$ which are contained in ${\Delta}_{n,i}$.

To prove this we use the adjunction formula which gives us that
\begin{equation} \label{equ:can1}
\wDelta_{n,i}|_{\wDelta_{n,i}} =K_{\wDelta_{n,i}} - K_{\wub(n)}|_{\wDelta_{n,i}}\,.
\end{equation}
Recall that $\wDelta_{n,i}$ is isomorphic to the blow-up of the universal elliptic curve $p_n: \ox[1](n)\to\oa[1](n)$ in the singular points of the fibers over the cusps. By the definition of the Hodge line bundle, we have
\begin{equation}\label{equ:can2}
K_{\ox[1](n)}=p_n^*(K_{\oa[1](n)}) + L
\end{equation}
and hence we obtain
\begin{equation}\label{equ:can3}
K_{\wDelta_{n,i}}=p_n^*(K_{\oa[1](n)}) + L + \sum_{\ell} E_{\ell}\,.
\end{equation}
Similarly, using the projection $q'_n: \oub(n)=\ox[1](n) \times_{\oa[1](n)}\ox[1](n) \to \oa[1](n)$ and the fact that  $\wDelta_{n,i}$ is a small resolution of $\oub(n)$, we find that
\begin{equation}\label{equ:can4}
K_{\wub(n)}= q_n^*(K_{\oa[1](n)}) +2L
\end{equation}
where $q_n:  \wDelta_{n,i} \to \oa[1](n)$ is the blow-down map followed by the projection $q'_n$.
The claim now follows from combining equations \eqref{equ:can1}, \eqref{equ:can3} and \eqref{equ:can4}.

One can now argue as follows. Since $L^2=0$ and $L.E_{\ell}=0$  on $\wDelta_{n,i}$, any  contribution to $\wDelta^3$ comes from triple intersections $E_i^3$  coming from the  singular points of  $\oub(n)$. Through any such point pass~$n$ components of~$\wDelta$. Any triple of these components (which can be the same) contributes $-1$ to the intersection number. In this way we obtain $-n^3\cdot n^2\,t(n)$ where $n^2\,t(n)$ comes from the number of singularities. Dividing by the order of the group $H_n$, which is $2n^5t(n)$, we obtain the result.
\end{proof}

Finally, we consider the Poincar\'e line bundle $P$  on $\ua[1] \times_{\ab[1] } \ua[1]$. By e.g.~\cite[Sec.~7]{ergrhu2}  the class of this line bundle is given by
\begin{equation*}
P=T_1 + T_2 - \wDelta.
\end{equation*}
We want to extend this to $\oub$ and $\wub$ respectively. Any two such extensions differ by a multiple of a fiber and we want to choose the extension in such a way that the restriction to the $0$-section is trivial.
For this we define
\begin{df}
We define the {\em Poincar\'e  class} on $\wub$ by
\begin{equation}\label{eq:Poincare}
P=T_1 + T_2 - \wDelta\ -L.
\end{equation}
\end{df}
With this definition we have indeed achieved that the restriction to the $0$-section is trivial, since
$$
 PZ_1Z_2=(T_1+T_2-\wDelta -L)Z_1Z_2=\frac{1}{24}-\frac{1}{24}=0 \,.
$$

We denote the closure of the antidiagonal in $\oub$ by $\wDelta^-$ and its strict transform in $\wub$ by $\wDelta^-$.
Note that the diagonal and the antidiagonal correspond to different rulings in the quadric cones of the
nodes of $\oub$.
Hence $\wDelta^-$ is isomorphic to $\Delta^-$ and intersects the exceptional line $E$ transversally in one point. For future use we note
\begin{lm}\label{lem:classPoincare}
The class of the antidiagonal is given by
\begin{equation}\label{eq:aDelta}
 \wDelta^-=T_1+T_2+P-L=2T_1+2T_2 - \wDelta -2L\,.
\end{equation}
\end{lm}
\begin{proof}
To express $\wDelta^-$ in terms of the Poincar\'e class, note that the first equality~\eqref{eq:aDelta} holds on the general fiber of $\wub$. To check that the factor of $L$ is correct it is enough to check the restriction of  $\wDelta^-$ to the $0$-section.
But the latter is equal to that of $\wDelta$ with the $0$-section. This follows from the observation
that we have an automorphism of $\oub$, namely the Kummer involution on one of the factors, which interchanges $\Delta$ and $\Delta^-$ and leaves the $0$-section invariant, together with the fact that the
$0$-section $Z_1\cap Z_2$ does not meet the singular point of $\oub$. The second equality follows immediately from~\eqref{eq:Poincare}.
\end{proof}

We will also need the
\begin{lm}\label{lm:EDelta} The intersection of $\wDelta^-$ on $\wub$ with the exceptional line $E$ is given by
\begin{equation}\label{equ:intersectionantidiagonalE}
E.\wDelta^-=\frac12.
\end{equation}
\end{lm}
\begin{proof}
We use a similar argument as in the computation of $\wDelta^3=-\frac12$ in \Cref{lm:Delta3}. We first fix a singularity of $\oub(n)$.
We recall that the stabiliser subgroup of a singularity of $\oub$ which acts trivially over the base is isomorphic to $\ZZ/n\ZZ \times \ZZ/n\ZZ$. In fact we should think of this as  $\mu_n \times \mu_n$ where $\mu_n$
is the group of $n$-th roots of unity. Every component~$\wDelta_{n,i}$ of the preimage~$\wDelta_n$ of~$\wDelta$ is fixed by a subgroup~$\mu_n$ of~$\mu_n \times \mu_n$, namely the subgroup $\{(\lambda, \lambda)| \lambda\in \mu_n\}$.
Similarly, every   component~$\wDelta_{n,j}^-$ of the preimage~$\wDelta_n$ of~$\wDelta^-$ is fixed by a subgroup  $\mu_n$ of $\mu_n \times \mu_n$, namely the subgroup $\{(\lambda, \lambda^{-1})| \lambda\in \mu_n\}$.
Fixing a pair $(\wDelta_{n,i},\wDelta_{n,j})$ gives rise to small resolution and an exceptional curve $E_{i,j}$ in $\wub(n)$. As we have already observed every component $\wDelta_{n,k}^-$ meets $E_{i,j}$ transversally in one point.
In this was we have $n^2$ pairs and each pair contributes $n$ so that in total we get $n^3$.
Since we have $n^2t(n)$ singularities in $\oub(n)$ we obtain a contribution $n^5t(n)$, which we have to divide by the order $|H(n)|=2n^5t(n)$ giving us $1/2$ as claimed.
\end{proof}

We must now put this into the context of the geometry of  $\oa$.
For this we recall that the toroidal compactification $\oa$ is made up of strata which correspond to the cones $\{\sigma\}$ in the second Voronoi fan, modulo the action
of $\GL(3,\ZZ)$. If $\sigma$ is a rank $i$ cone, then we associate to it a torus bundle $\calT(\sigma)$ over the fiber product ${\calX}_{3-i}^i$. The dimension of $\calT(\sigma)$ is $6 - \dim (\sigma)$. Let $G(\sigma)$ be the group
which is the stabilizer of $\sigma$ in its $\QQ$-span. Then $G(\sigma)$ acts on~$\calT(\sigma)$ and the quotient is the stratum $\sigma \subset \oa$.

In our case, the stratum $\osi{K3}$ admits the following description
\begin{lm}\label{lm:osiK3geometry}
The group $G(\si{K3})$ acts on $\wub$ and there is a finite  $G(\si{K3})$-equivariant morphism $f(\si{K3}): \wub\to \osi{K3}$ which, in particular, identifies the quotient of
 $\ua[1] \times_{\ab[1]} \ua[1]$ by $G(\si{K3})$ with $\si{K3}$.
\end{lm}

Before we give the proof, we first recall the group  $G(\si{K3}) \subset \Sp (4,\ZZ)$ which was computed in \cite[p. 101]{huto1}. It is generated by  the involutions
\begin{equation}\label{equ:generatorsGK3}
\begin{array}{lll}
(x_1,x_2) &\mapsto& (-x_1,-x_2) \\
(x_1,x_2) &\mapsto& (x_2,x_1) \\
(x_1,x_2) &\mapsto& (x_1,x_1-x_2)\,.
\end{array}
\end{equation}
and from this we see that $G(\si{K3})\cong S_3 \times \ZZ/2\ZZ$.

\begin{proof}
By construction of the toroidal compactification, the open stratum can be described as $\si{K3} \cong (\ua[1] \times_{\ab[1]} \ua[1])/G(\si{K3})$. We now have to prove that the action of $G(\si{K3})$ extends to an action on $\wub$, and that the
morphism  $\ua[1] \times_{\ab[1]} \ua[1] \to {\beta}(\si{K3})$ extends to~$\wub$. Having done this, the $G(\si{K3})$-equivariance will be automatic.

We first show that the group action extends. This is obvious for the  maps $(x_1,x_2) \mapsto (-x_1,-x_2)$ and $(x_1,x_2) \mapsto  (x_2,x_1)$. The first involution  is just the Kummer involution, which we have already divided out in the construction of $\ox[1]$, and the second involution interchanged the factors of  $\ox[1] \times_{\oa[1]} \ox[1]$. In both cases these involutions are compatible with the small resolution. It remains to treat the involution $\iota:(x_1,x_2) \mapsto (x_1,x_1-x_2)$. The action on a general fiber $E \times E$ of $\ua[1] \times_{\ab[1]} \ua[1] \to \ab[1]$ is given by the dual action and this was computed in \cite[p. 99]{huto1} to be given by $(x,y) \mapsto (x+y,-y)$.

Our approach is the following: both varieties $\oub$ and $\osi{K3}$ are, up to finite group actions, toroidal varieties, as is the small resolution $\wub$. We shall compare the toroidal structures of these varieties and use toric geometry
to show that the map extends. For this we shall work on a level cover. In fact we will work on a level cover $\oub(n)$ with odd $n \geq 3$ and we shall choose a singularity $(P,P)$ where $P$ is a singular point of a fiber
of the universal elliptic curve $\ox[1](n) \to \oa[1](n)$ over a cusp such that $P$ is fixed under the Kummer involution of $\ox[1](n)$ which, near $P$, interchanges the two components of the $n$-gon which meet in $P$.

We first recall how the compactification of  $\ua[1](n)\to\ab[1](n)$
to  $\ox[1](n) \to \oa[1](n)$ is obtained near $P$.
Recall that $\ua[1](n)$ is the quotient of the product  $\CC \times \HH_1$. We shall use  $z,\tau$ as the coordinates on this product and put
$u:=e^{2\pi n iz}$ and $t:=e^{2\pi in\tau}$.  We interpret these as coordinates of a torus $T_{\{0\}}=(\CC^*)^2$.
The construction of the universal elliptic curves $\ox[1](n) \to \oa[1](n)$ is explained in detail in \cite[Sec.~2B]{hukawebook}. The main idea is to define embedding $T_{\{0\}} \to T_k=\CC^2$ for $k\in \ZZ$ and then divide out $\cup_k T_k$ by an action of $\ZZ$ which cyclically maps $T_k$ to $T_{k+n}$.
 Let us consider the embedding
 \begin{equation*}
 i_0: T_{\{0\}}  \to T_0, \quad  (z,t) \mapsto (z,z^{-1}t)=:(a,b).
 \end{equation*}
 Locally the toroidal compactification is given by adding the two coordinate axes $a=0$ and $b=0$ which give rise to (open parts) of the two lines in the $n$-gon meeting at $P$.
 In our situation we have two copies of $\ua[1](n)\to\ab[1](n)$ and we use the coordinates $(z,t)$ and $(\tilde z, \tilde t)$, respectively  $(a,b)$ and $(\tilde a, \tilde b)$, on the two factors .
 The equation of the fiber product then becomes $ab - \tilde a \tilde b=0$ verifying that the fiber product ia a $3$-fold with an isolated $A_1$-singularity near the point $(P,P)$.

The next step is to understand the action of the Kummer involution $\iota$ on $\ox[1](n) \to \oa[1](n)$ near the point $P$. This is discussed in detail in \cite[Sec,~2B, p.~31]{hukawebook}.
Recall the torus embeddings $T_0 \to T_{k}, (u,v) \mapsto (uv^{-k}, u^{-1}v^{k+1})$. Then $\iota$ is given by $\iota_k: T_k \to T_{-k-1}$ with
$(u_k,v_k) = (s,t) \mapsto (u_{-k-1}, v_{-k-1})=(t,s)$  and then using
$h_{2k+1}: T_{-k-1} \to T_{k}$ with $(u_{-k-1},v_{-k-1})=(s,t)  \mapsto (u_{k},v_{k})=(s,t)$. This translates into $\iota: ( a,  b)\mapsto ( b,  a)$. Indeed, this is intuitively clear as the involution fixes the origin and
interchanges the two coordinate axes.

This now allows us to understand the action of $\iota$ near $(P,P)$. Since this acts by $(x,y) \mapsto (x+y,-y)$ on a general fiber $E \times E$ this can now becomes
$$
((z, \frac{t}{z}), (\tilde z, \frac{ t}{\tilde z})) \mapsto ((z\tilde z,\frac{t}{z\tilde z}),( \frac{ t}{\tilde z}, \tilde z)).
$$
This can be written as
$$
\iota: ((a,\tilde a),(b, \tilde b)) \mapsto (a\tilde a, \frac{\tilde b}{a}),(\tilde b, \tilde a))=  (a \tilde a, \frac{b}{\tilde a}),(\tilde b, \tilde a)).
$$
This is clearly well defined outside $a = \tilde a=0$. We note, however that even if $a=\tilde a=0$, as long as $\tilde b\neq 0$, this is well defined on $\oub(n)$. Indeed, the point $(0, \infty)$ is the point $\infty$
on the $\PP^1$ in our $n$-gon  which is the compactification of the coordinate line $a=0$.  The same applies to $b \neq 0$.  Hence the map is well defined outside the node.
The indeterminacy of the map $\iota$ can be resolved by a small resolution of the $A_1$-singularity, namely the blow-up along the Weil divisor
$a=\tilde b=0$, which is in fact the same as the blow-up along  $\tilde a = b=0$. We note that this also the same as the small resolution given by blowing up $a - \tilde a= b - \tilde b = 0$, which is the blow-up of the diagonal.

Next we recall that the small resolutions of a $3$-dimensional quadric cone are toric.
Let $\sigma$ be the cone spanned by four vectors $v_i \in \ZZ^3, i= 1, \dots 4$ such that any three of them form a basis of $\ZZ^3$ and $v_1+v_4=v_2+v_3$. This gives rise
to the quadric cone
$$
Q: x_1x_4 - x_2x_3=0.
$$
The two small resolutions of $Q$ are given by either blowing up a Weil divisor in one of the two families of planes contained in $Q$ and planes in the same family define the same small resolution.
From a toric point of view this corresponds to the two subdivision of $\sigma$ into two basic cones. These can be obtained by either adding the cone $\tau= \RR_{\geq 0}v_1 + \RR_{\geq 0}v_4$ or
the cone $\tau= \RR_{\geq 0}v_2 + \RR_{\geq 0}v_3$. A straightforward toric computation shows that the first possibility corresponds to blowing up $x_1=x_3=0$ (or equivalently $x_2=x_4=0$).
The second choice corresponds to blowing up $x_3=x_4=0$ (or equivalently $x_1=x_2=0$).  Identifying the coordinates $x_1,x_4,x_2,x_3$ with $a,b,\tilde a, \tilde b$ we are exactly in the situation where we
add the cone $\tau= \RR_{\geq 0}v_1 + \RR_{\geq 0}v_4$.

We must now check that the morphism  $f(\si{K3})(n): \ua[1](n) \times_{\ab[1](n)} \ua[1](n) \to {\beta}(\si{K3})(n)$ extends to~$\wub(n)$.
For this we compare the toric compactification of the universal elliptic curve with the toric compactification $\osi{K3}(n)$ of
the stratum $\sigma_{K3}(n)$ in $\oa[3](n)$.

For this we have to understand the compactification of the stratum associated to $\sigma_{K3}(n)$. We shall again work on a level cover. To describe the
compactification we have to understand the part which is added over the cusp $\ab[0]$.  Analytically locally near the boundary, $\osi{K3}(n)$ is a $3$-dimensional toric variety. The torus in question is $\Sym^2(\ZZ^3) \otimes \CC^*$
divided by the subtorus associated to the three-dimensional space spanned by $\sigma_{K3}$, i.e. by $x_1^2,x_2^2,(x_1-x_2)^2$. We can read off the fan used for the toroidal  compactification from the
cones listed in Table \ref{table:cones}. This is given, again modulo the subspace spanned by $\sigma_{K3}$,  by the cones
$\sigma_{K4}$ and $\iota(\sigma_{K4})$ which in our quotient space are generated by $\langle x_3^2,(x_1-x_3)^2, (x_2-x_3)^2 \rangle $ and  $\langle x_3^2,(x_1-x_3)^2, (x_1-x_2-x_3)^2 \rangle $, together with their faces.
Labelling the generators $x_3^2,(x_1-x_3)^2, (x_2-x_3)^2, (x_1-x_2-x_3)^2$ by $v_1,w_1,v_2,w_2$ we are exactly in the situation of the above one $\sigma$, as $v_1+w_1=v_2+w_2$ (modulo the space generated by
 $x_1^2,x_2^2,(x_1-x_2)^2$). We are exactly in the situation where  $\sigma$ is
subdivided into two basic cones by introducing the cone $\RR_{\geq0}v_1 + \RR_{\geq0}w_1$.
Hence we can identify the two fans used to compactify $\si{K3}$ and $\ua[1] \times_{\ab[1]} \ua[1]$.
It remains to identify the two $3$-dimensional tori used in the toroidal compactifications: the result then follows from the standard relationship between maps of fans and morphisms of toric varieties.
But this is easy. As explained above, the torus used in our description of  $\osi{K3}(n)$ has the natural coordinates $t_{33}=e^{2\pi in\tau_{33}}, t_{13}=e^{2\pi in\tau_{13}}, t_{23}=e^{2\pi in\tau_{23}}$ (where $n$ again comes from
the level structure). Here $\tau=(\tau_{ij}) \in \HH_3$ are the usual coordinates in Siegel space. The coordinates for $\ua[1](n) \times_{\ab[1](n)} \ua[1](n)$ are $\tau_{33}$ for the base and $\tau_{13}, \tau_{23}$ for the fibers. The coordinates of the torus
used in the local compactification of $\wub(n)$ is again given by taking the exponentials. Hence the two tori can naturally be identified and this map extends to a local isomorphism of $\wub(n)$  and $\osi{K3}(n)$.
This concludes the proof.
\end{proof}
\begin{rem}
This proof shows that the requirement to extend the  action of the group $G(\si{K3})$ necessarily leads us to the small resolution of $\ox[1] \times_{\oa[1]}\ox[1]$.
\end{rem}
\begin{rem}\label{rem:stacky12}
In our above calculations of intersection numbers we have already taken the involution $(x_1,x_2) \mapsto  (-x_1,-x_2)$ into account, but not yet however the action of the symmetric group $S_3$.
There is also a further subtlety which will play a role several times. This is that the stabilizer subgroup of a cone depends on the genus where we consider it. If we move from genus $2$ to genus $3$,
the order of this group increases by a factor of~$2$. The reason is the
involution given by $-1 \in \Sp(6,\ZZ)$.  Note that the involution $(x_1,x_2) \mapsto  (-x_1,-x_2)$ in $\Sp(4,\ZZ)$ and $-1 \in \Sp(6,\ZZ)$, differ in so far as they  come  from the matrices
\begin{equation*}
\begin{pmatrix} -1& 0& 0\\ 0 &- 1&0\\ 0&0&1\end{pmatrix} \qquad\hbox{and}\qquad \begin{pmatrix} -1& 0& 0\\ 0 &- 1&0\\ 0&0&-1\end{pmatrix}
\end{equation*}
respectively. In total this means that we have to multiply  the numbers computed above on $\tilde\calY_1$ by  a further stacky factor of $1/12$.
\end{rem}

We want to conclude this section with a description of the cohomology of~$\osi{K3}$.

\begin{lm}\label{lm:cohY}
For the  real cohomology groups we have
\begin{equation*}
H^2(\osi{K3}, \RR) \cong H^4(\osi{K3}, \RR) \cong \RR^2.
\end{equation*}
In terms of invariant classes on  $\wub$ these groups are given by
\begin{equation*}
H^2(\osi{K3}, \RR) = \RR L\oplus\RR (Z_1+Z_2+\wDelta^-)
\end{equation*}
and
\begin{equation*}
 H^4(\osi{K3}, \RR)=\RR (4T_1T_2-P^2) \oplus\RR (Z_1+Z_2+\wDelta^-)L.
\end{equation*}
\end{lm}
\begin{proof}
We first observe that  Poincar\'e duality holds since $\osi{K3}$ has only finite quotient singularities. Moreover, the cohomology of $\osi{K3}$ is the $G(\si{K3})$-invariant cohomology
of $\wub$. We recall that $H^2(\wub,\RR)$ is $4$-dimensional
and generated by the classes $L,Z_1,Z_2$ and $\wDelta^-$ (where we could have replaced $Z_1,Z_2$ by $T_1,T_2)$. In the proof of \Cref{lm:osiK3geometry} we discussed the action of $G(\si{K3})$ on these classes.
Indeed, the three generators of this group described in Equation \eqref{equ:generatorsGK3} act as follows. The first automorphism acts trivially, the second interchanges $Z_1$ and $Z_2$ and leaves $L$ and $\wDelta^-$ invariant, whereas the third
leaves $L$ and $Z_2$ invariant and interchanges $Z_1$ and $\wDelta^-$. Hence the invariant subspace is $2$-dimensional. This gives the Betti numbers. Since the given classes are invariant and independent, the remaining claims follow.
\end{proof}

The surface $\osi{K3+1}$ is the closure of a stratum that is contained in $\osi{K3}$. We will thus consider its preimage $f(\si{K3})^{-1}(\osi{K3+1})\subset \wub$ under the finite map described in \Cref{lm:osiK3geometry} to compute its intersection numbers. Indeed, in terms of the geometry of $\wub$, the stratum $\osi{K3+1}$ is simply the fiber over the cusp of the map $\wub\to\oa[1]$. Thus the class of the surface $\osi{K3+1}$ on $\wub$ is the pullback of the class of the point $i\infty=\partial\oa[1]$, which is the class of the fiber $F=12L$.

To perform the intersection computations on $\wub$, we will want to identify geometrically the pullbacks of the classes in $H^4(\oa)$ to $\osi{K3}$ under inclusion, further pulled back to $\wub$ under $f(\si{K3})^*$. The pullback of $L$ is of course still $L$, while for the boundary we will now follow the approach and use the computations of~\cite{ergrhu2}, which we now recall. For this we will work on the level $2$ cover~$\oa{}(2)$ of~$\oa$, and recall that the preimage of the irreducible boundary divisor $D\subset\oa$ there is the union $\cup D_i$ of $2^{2g}-1$ irreducible components indexed by points of $(\ZZ/2\ZZ)^{2g}\setminus\lbrace 0\rbrace$.

We have already discussed the triple intersections of the form $D_i \cap D_j \cap D_k$ and
that the image of such an intersection  under the map $\oa(2) \to \oa$ can either lie over  ${\Sat[2]}$ or over ${\Sat[1]}$, depending on whether we are in the {\em local} or {\em global} case, these in turn corresponding to the
cones $\si{K3}=\langle x_1^2,x_2^2,(x_1-x_2)^2\rangle$  and $\si{1+1+1}= \langle x_1^2,x_2^2, x_3^2\rangle$. As before we will write these intersections as $D_i \cap D_j \cap D_{i+j}$ or $D_i \cap D_j \cap D_k$ respectively.
With this notation each irreducible component $Z$ of the preimage of~$\osi{K3}$ in~$\oa{}(2)$ is the intersection of a triple of irreducible boundary components of the form $D_i\cap D_j\cap D_{i+j}$. Here we use the following geometric interpretation:
The divisors $D_i$ and $D_j$ correspond to the cone $\si{1+1}=\langle x_1^2,x_2^2 \rangle$ and the corresponding stratum is a $\CC^*$-bundle over $\ua[1] \times_{\ab[1]} \ua[1]$ which is compactified first to a $\PP^1$-bundle and
then extended to the cusp. The $\PP^1$-bundle is the $\PP^1$-bundle which compactifies the Poincar\'e bundle by adding the section at infinity. Intersecting with~$D_{i+j}$ then cuts out the $0$-section of the $\PP^1$-bundle.
Using this convention we identify the restrictions of boundary divisors of~$\oa{}(2)$ to~$\osi{K3}$, pulled back to $\wub$.
We note that these equalities on $\calY$ were proven in~\cite{ergrhu2}, while we now claim that they also hold on $\wub$.
\begin{lm}\label{lem:restrictionDitoZ}
The following holds:
\begin{equation}
 D_i|_{\wub}=-2T_1-P, \, D_j|_{\wub}=-2T_2-P, \, D_{i+j}|_{\wub}=P, \, \sum_{k\ne i,j,i+j} D_k|_{\wub}=12L.
\end{equation}
\end{lm}
\begin{proof}
We work on $\oa(2)$ and denote by $Z$ some component of the preimage of $\osi{K3}$. Note that $Z$ is the image of a map
$f(\si{K3})(2):   \wub(2) \to \oa[3](2)$, which covers the map $f(\si{K3}):   \wub \to \oa[3]$ from \Cref{lm:osiK3geometry}. We also denote by $Z^0$ the open part of $Z$ which is the preimage of the open stratum
$\si{K3}$. We first notice that the union of the boundary divisors $D_k$, for all $k\neq i,j,i+j$ cut out the fiber $\osi{K3+1}$ over the cusp of $\oa[1]$ and hence
$$
\sum_{k\ne i,j,i+j} D_k|_Z=12L.
$$
By~\cite[Lem.~5.4]{ergrhu2} we further have:
$$
 D_i|_{Z^0}=(-2T_1-P)|_{Z^0}, \, D_j|_{Z^0}=(-2T_2-P)|_{Z^0}, \, D_{i+j}|_{Z^0}=P|_{Z^0}.
$$
We claim that these equalities hold on all of~$Z$. We shall prove this in the first case, the others being the same. Since the difference between $Z$ and $Z^0$ is the fiber over the cusp, it follows from the above equation that
$$
D_i|_Z=(-2T_1-P + \alpha L)|_Z\,.
$$
We claim that $\alpha = 0$. Indeed, by \cite[Prop.~5.1]{ergrhu2}, see also \cite[Prop.~3.2]{huleknef} we have the equality
$$
D_i|_{D_i \setminus \beta_3(2)} = -\Theta_{D_i \setminus \beta_3(2) }
$$
in $\operatorname{CH}(\oa(2)\setminus \beta_3(2))$,
where we recall that $D_i \to \oa[2](2)$ is the universal abelian variety (which is the universal Kummer variety in level $2$), and $\Theta$ is the universal theta divisor extended to Mumford's partial compactification and trivialized along the $0$-section.
Since the intersection of $\beta_3(2)$ with the
$0$-section of the universal family (which is isomorphic to $\oa[2]$) is the fiber over~$\ab[0]$, and hence has codimension $2$, it follows that the
restriction of $D_i$ to the entire $0$-section is trivial. Since also~$T_1$ and~$P$ are trivial over the $0$-section it follows that
$\alpha=0$.
 \end{proof}
 \begin{rem}\label{rem:permutations}
We note that the divisors $-2T_1-P$, $-2T_2-P$ and $P$ are all permuted under the action of the group $G(\si{K3})$.
\end{rem}
We can now compute the homology class of the surface $\osi{K3+1}$, reproving \Cref{lm:classosiK3+1}.
\begin{proof}[Alternative proof of \Cref{lm:classosiK3+1}]
We have already remarked that the class of $\osi{K3+1}$ is the class of the fiber $\osi{K3+1}=F=12L\in H^2(\wub)$ over the cusp $\ab[0]$ and hence the restriction of $L$ to  $\osi{K3+1}$ is trivial. This gives, in particular,  immediately the first three intersection numbers being zero.

Using the convention explained above, we can, using \eqref{eq:Poincare}, now compute:
\begin{equation}\label{eq:DK3}
\begin{aligned}
 D|_{\wub}&=\sum D_a|_{\wub}=\left(D_i+D_j+D_{i+j}+\sum_{k\ne i,j,i+j} D_k\right)|_{\wub}\\ &=-2T_1-P-2T_2-P+P+12L=-3T_1-3T_2+\wDelta+13L\,,
\end{aligned}
\end{equation}
where we worked with the pullback from $\osi{K3}$ to $\wub$, furthermore pulled back to the level $2$ cover.

We note that of course as $D|_{\wub}$ is the pullback of $D|_{\osi{K3}}$, by \Cref{lm:cohY} its class must be a linear combination of the classes $L$ and $Z_1+Z_2+\wDelta^-$, and the above formula identifies this linear combination.

The computation of the restriction of the class $\beta_2$ to $\osi{K3}$ turns out to have one crucial subtlety, and we will treat this as separate \Cref{lm:beta2K3} below.

To finish the proof of the Proposition, we use the intersection numbers computed in \Cref{prop:tildeY}. Taking also the factor $1/12$ into account as discussed in \Cref{rem:stacky12},  and using repeatedly~\eqref{eq:Poincare} and that $L^2|_{\wub}=0$, we compute:
$$
\begin{aligned}
12D^2.\osi{K3+1}&=12L(-3T_1-3T_2+\wDelta + 13 L)^2 \\ & = 12L(9T_1^2+9T_2^2+\wDelta^2+18T_1T_2-6(T_1+T_2)\wDelta)\\
 & =12\cdot (0+0+0+18-6-6)\cdot\frac{1}{24}=12\cdot 6\cdot\frac{1}{24}=3.
 \end{aligned}
$$
Similarly
$$
\begin{aligned}
 12\beta_2.\osi{K3+1}&=12L(4T_1T_2-P^2+12L(-3T_1-3T_2+\wDelta))\\
 &=12L(4T_1T_2-(2T_1T_2-2\wDelta(T_1+T_2))\\&=12\cdot (4-(2-4))\cdot\frac{1}{24}=3\,.
\end{aligned}
$$
The equality $M^2. 12\osi{K3+1}=1/4$ follows immediately from the definition of $M$ and the numbers already computed.
This proves the Proposition.
\end{proof}
\begin{lm}\label{lm:beta2K3}
The pullback to $\wub$ of the restriction of $\beta_2$ to $\osi{K3}$ is given by
\begin{equation}\label{eq:beta2K3}
 \beta_2|_{\wub}=4T_1T_2-P^2-6L(3T_1+3T_2-\wDelta)\,.
\end{equation}
\end{lm}
\begin{proof}
To perform this computation, we go to the level cover $\oa{}(2)$ and use the same notation as for the derivation of~\eqref{eq:DK3}. The subtlety
here is the exceptional divisor $E$ which comes from the small blowup of $\oub$.

Again using \eqref{eq:Poincare} and the fact that $L^2=0$ we find that
$$
\begin{aligned}
 \beta_2|_{\wub}=&\sum_{a<b} D_aD_b|_{\wub}\\&=(-2T_1-P)(-2T_2-P)+(-2T_1-P)P+(-2T_2-P)P  \\ &\ +12L(-2T_1-P-2T_2-P+P) + \sum_{a<b,\ a,b \notin \{i,j,i+j\}} D_aD_b&\\
 &=(4T_1T_2-P^2)+12L(-3T_1-3T_2+\wDelta)\, + E.
 \end{aligned}
$$
Here we think of the codimension two locus $E \in H_2(\wub)\simeq H^4(\wub)$, via Poincar\'e duality.
Note that  this fits with \Cref{lm:cohY}  where we have shown that  the first two summands of this sum generate $H^4(\wub,\RR)$.
(We also refer the reader to~\cite{grza} for a discussion of why the class $4T_1T_2-P^2$ should naturally appear here).

It thus remains to determine the class of $E$. By \Cref{lm:cohY} and the fact that $E.L=0$ it follows that
$E$ is a multiple of $L(Z_1+Z_2 +\wDelta^-)$, or, using \eqref{eq:aDelta} and $L^2=0$, equivalently of $L(3T_1+3T_2-\wDelta)$.
By \Cref{lm:EDelta} the intersection number $E.\wDelta^-=1/2$, and a straightforward calculation then gives
\begin{equation}\label{equ:classofE}
 E=6L(3Z_1+3Z_2-\wDelta)\,.
\end{equation}
and thus the claim.
\end{proof}
Note that the only intersection number involving $\beta_2$ that we have so far computed is $\beta_2.\osi{K3+1}$. As $\osi{K3+1}=12L$ on $\osi{K3}$, this intersection number in particular is insensitive to adding any class to $\beta_2$ of the form $L.x$, and in particular does not notice the subtlety regarding the exceptional locus~$E$ appearing in $\beta_2$, as described above. As these intersection computations on $\osi{K3}$ become very laborious, to avoid mistakes we have computed the intersection numbers from \Cref{prop:tildeY}, the expression~\eqref{eq:Poincare} for the class of the Poincar\'e bundle, and the above formulas~\eqref{eq:DK3} and~\eqref{eq:beta2K3} for the restrictions of $D$ and $\beta_2$ to $\wub$, using Maple. We then verified that
$$
\sigma_1^3(\sigma_3-\beta_3)=\frac{D|_{\wub}^3}{12}=\frac{31}{48}
$$
and
$$
\sigma_1\sigma_2(\sigma_3-\beta_3)=\frac{D|_{\wub}\beta_2|_{\wub}}{12}= \frac{5}{16},
$$
where on the left-hand-side these are the notation of~\cite{vdgeerchowa3}, so that $\sigma_1=D$, $\sigma_2=\beta_2$, and $\osi{K3}=\sigma_3-\beta_3$, as discussed above, and the factor of~$\frac{1}{12}$ comes from the group described in \Cref{lm:osiK3geometry}. We note that these computations provide a highly non-trivial check for our description of the intersection theory on $\wub$ and for the computation of classes $D|_{\wub}$ and $\beta_2|_{\wub}$ in~\eqref{eq:DK3} and~\eqref{eq:beta2K3}. While the computation of the class of the surface $\osi{K3+1}$ above was insensitive to any class in $\beta_2|_{\wub}$ that is a multiple of~$L$ (simply because $L^2=0$), these computations matching the numbers computed by van der Geer do involve all the summands of the class $\beta_2|_{\wub}$.

We conclude this section by recomputing the class of the surface $S_D$, thereby also shedding new light on a different interpretation the geometry of this surface.
For this we consider in $\wub$ the union of the preimages of the two $0$-sections on the two factors of $\ox[1]\times_{\oa[1]}\ox[1]$ and the anti-diagonal. These three surfaces are permuted by the group $G(\si{K3})$ and we denote their
image in $\oa$ by $S'_D$. By inspection the surface $S'_D$ parameterizes semi-abelic threefolds which are the product of a fixed elliptic curve with a singular semi-abelic surface. Hence the surfaces $S'_D$ and $S_D$ coincide, but
bearing the different construction of $S'_D$ in mind, we retain this notation for the rest of this section.

We will now use the elaborately described geometry of~$\osi{K3}$ to compute the intersection numbers of $S'_D$ with the pullback of our classes forming a basis of  $H^4(\oa)$ and thus reconfirm our computation of the class
of $S_D$ in  \Cref{prop:intersections}. This serves as an elaborate check on our numerics, and the techniques we use.
\begin{prop}\label{prop:classS'D}
The surface $S_D'$ has the following intersection numbers on $\oa$:
\begin{equation*}
 L^2.S_D'=0, \, \, LD.S_D'=-\frac{1}{48}, \, D^2.S_D'= - \frac{11}{24}, \, \beta_2.S_D'=-\frac18,
\end{equation*}
\begin{equation*}
 LM.S_D'= \frac{1}{48}, \,M^2.S_D'=\frac{1}{24}.
\end{equation*}
\end{prop}
\begin{proof}
We recall that by definition the preimage in $\wub$ of $S_D'$ is the union of the two $0$-sections $Z_1,Z_2$, and the antidiagonal $\wDelta^-$. Equalities~\eqref{eq:Poincare},~\eqref{eq:aDelta}, and $T_i=Z_i+L$ therefore give
\begin{equation}\label{equ:classofS0}
S_D'=Z_1 + Z_2 + \wDelta^- = 3T_1+3T_2-\wDelta -4L\, .
\end{equation}

The class $L^2$ is zero on $\osi{K3}$ and this immediately gives $S_D'.L^2=0$.

We recall that the classes of $D$ and $\beta_2$ restricted to $\osi{K3}$ are given by~\eqref{eq:DK3} and~\eqref{eq:beta2K3}, respectively.
Using \Cref{prop:tildeY} and remembering to divide by $12$, as explained in \Cref{rem:stacky12}, the computation of the intersection numbers on $\osi{K3}$ is now purely mechanical.
To avoid arithmetic errors we have determined the intersection theory on $\wub$ using a small Maple program. The results are:
$$
LD.S_D'=\frac{1}{12}L(-3T_1-3T_2+\wDelta+13L)(3T_1+3T_2-\wDelta -4L)=- \frac{1}{48}
$$
and
$$
D^2.S_D'=\frac{1}{12}(-3T_1-3T_2+\wDelta+ 13L)^2(3T_1+3T_2 -\wDelta -4L)=- \frac{11}{24}\,.
$$
Finally, we have
 $$
  \beta_2.S_D'=\frac{1}{12}(4T_1T_2-P^2- 6L(3T_1 + 3T_2-\wDelta)\cdot(3T_1+3T_2 -\wDelta-4L) =-\frac18\,.
$$
Form this we immediately obtain
$$
LM. S_D'=L(12L-D)S_D'.L=-LD.S_D'=\frac{1}{48}
$$
and
$$
M^2.S_D'=(12L-D)^2.S_D'=-24LD.S_D'+D^2.S_D'=\frac{1}{24}\,.
$$
\end{proof}

\section{Extremal effective rays of $\Eff(\oa)$}\label{sec:extremality}
In this section we prove our main theorem, showing that all five surfaces $S_A,S_F,S_D,\osi{C4},\osi{K3+1}$ are indeed extremal effective rays of the cone $\oE(\oa)$.
Recall that it does not immediately follow that the classes of these five surfaces
generate  $\oE(\oa)$. Nevertheless, we conjecture that they in fact do generate this cone, and give some further evidence for this claim in \Cref{rem:evidence}. The arguments for the extremality are heavily dependent on the fact that the classes $L^2,LM,M^2$ are known to be nef. We furthermore continue to use the fact that the contraction $\pi:\oa\to\Sat$ to the Satake compactification is given by a multiple of $L$, so that $L$ is in fact a pullback from~$\Sat$.

Crucially, we use another morphism that was studied in the work of Shepherd-Barron~\cite{shepherdbarron}, and further investigated by him in detail in~\cite{sbexceptional}. This is the Kummer morphism: geometrically, given a principally polarized abelian variety, the morphism sends it to the corresponding Kummer variety, thought of for example as an element of $CH^*(\PP^7)$. Thinking this way, it is a priori unclear how this morphism extends to the compactification, and we refer to~\cite{perna} for a detailed study of this question. The way we will think about this morphism is following~\cite{sbexceptional}, and reliant on the main theorem from there:

\begin{thm}[Shepherd-Barron {\cite[Thm. 1.1(1)]{sbexceptional}}] \label{thm:sb}
The exceptional locus of the line bundle $M$ is equal to $N(\oa[1]\times\oa[2])\subset\oa[3]$.
\end{thm}
We recall here that by definition the exceptional locus of a nef line bundle $M$ on a variety $X$ is the union of all subvarieties $Z\subset X$ such that $M^{\dim Z}.Z=0$. In particular, by definition any subvariety $Z\subset X$ that is contracted to a smaller-dimensional variety under the map given by the linear system~$|mM|$ is contained in the exceptional locus. In this terminology, the exceptional locus of the nef line bundle $L$ on $\oa$ is equal to $\beta_1=\partial\oa$.
To avoid confusion, we note that in~\cite{sbexceptional} the notation $H_g$ is used for what we call $L$, and the notation $L_g$ is used for what we call $M$.

\medskip
We first prove the following auxiliary statement.
\begin{lm}\label{lm:E111}
The cone of effective surfaces on $\osi{1+1+1}$ is
$$
  \oE(\osi{1+1+1})=\Eff(\osi{1+1+1})=\RR_{\ge 0}\osi{K3+1}+\RR_{\ge 0}\osi{C4}\,.
$$
\end{lm}
\begin{proof}
We will prove that the effective cone $\Eff(\osi{1+1+1})$ is generated by~$\osi{K3+1}$ and~$\osi{C4}$, and it will thus follow that the pseudo-effective cone is the same.

To accomplish this, we apply the argument used by Shepherd-Barron in~\cite[Cor.~2.5]{shepherdbarron} to deal with effective curves. Indeed, recall that $\osi{1+1+1}$ is a toric variety
with an open dense orbit equal to $(\CC^*)^3$.
We thus act by $(\CC^*)^3$ on the Chow scheme of effective curves contained in $\osi{1+1+1}$.
As in~\cite[Cor.~2.5]{shepherdbarron}, Borel's fixed point theorem then guarantees that there is a fixed point of $(\CC^*)^3$ on this Chow scheme. Thus any Chow class of an effective surface must be represented by a surface (which may be reducible) that is invariant, as a set, under the action of $(\CC^*)^3$.
Since the open torus orbit $\si{1+1+1}$ is three-dimensional, this means that every homology class of an effective surface in~$\osi{1+1+1}$ can be represented by an effective surface that is disjoint from $\si{1+1+1}$. But then such a surface must lie in $\osi{1+1+1}\setminus\si{1+1+1}=\osi{K3+1}\cup\osi{C4}$ (where we used the description of the closure from~\eqref{eq:strataclosures}). As both of these are irreducible effective surfaces, it follows that they generate the cone $\Eff(\osi{1+1+1})$. Finally, from the intersection numbers in~\eqref{eq:intersect} it follows that the classes of $\osi{K3+1}$ and $\osi{C4}$ are linearly independent in $H_4(\oa)$, and thus they are also linearly independent in $H_4(\osi{1+1+1})$ (this could also be seen by examining the spectral sequence calculations in~\cite{huto1} to ascertain that $H^2(\osi{1+1+1})=\RR^2$).
\end{proof}

\medskip
We now start proving extremality of our surfaces.
\begin{prop}\label{prop:C4K31extremal}
The surfaces $\osi{C4}$ and $\osi{K3+1}$ are extremal effective rays of $\Eff(\oa)$.
\end{prop}
\begin{proof}
The argument turns out to be completely analogous for both surfaces, as it uses the fact that their intersection numbers with $L^2$ and $LM$ are equal to zero. We give the argument for the case of~$\osi{C4}$.
Suppose for contradiction that one can write $\osi{C4}=\sum a_iS_i$ for some effective surfaces $S_i$ and coefficients $a_i>0$. Since $L^2$ and $LM$ are nef classes, their intersection with any effective surface is non-negative. Since $L^2.\osi{C4}=LM.\osi{C4}=0$, it thus
follows that $L^2.S_i=LM.S_i=0$ for all $i$. From the intersection numbers in \Cref{prop:intersections} it thus follows that for any $i$ the class of the surface $S_i$ lies in the linear span of $\osi{C4}$ and $\osi{K3+1}$, i.e.~$S_i=a\osi{C4}+b\osi{K3+1}$. We thus need to prove that $a$ and $b$ are both non-negative,

Since $L$ is an ample class on the Satake compactification, it follows that the morphism $\pi:\oa\to\Sat$ must contract $S_i$ to at most a curve.
Recalling that $M=12L-D$, and that $-D$ is $\pi$-ample, it follows that if $\dim\pi(S_i)=1$, so that $L.\pi(S_i)>0$, then also $-LD.S_i>0$; since $L^2.S_i=0$, it would then also follow
that $LM.S_i>0$, which contradicts the previously observed equality $LM.S_i=0$. Thus~$\pi(S_i)$ must be zero-dimensional, i.e.~$S_i$ must be contained in a fiber of $\pi$. If $S_i$ is not contained in $\beta_2$, then it must be the entire fiber of the map $\pi$ over some point of $\ab[2]$. Such a fiber has zero intersection with classes $L^2$ and $LM$, as computed above, and
also zero intersection with the class $\beta_2$, since it is simply disjoint from that locus.
Thus from the intersection table in \Cref{prop:intersections} it follows that the class of such a fiber $S_1$ is a positive multiple of $4\osi{C4}+\osi{K3+1}$, in which case both coefficients~$a$ and~$b$ are indeed positive.

Thus the only remaining possibility is for $S_i$ to be contained in $\beta_2$, and by taking further irreducible components we can additionally assume that~$\pi(S_i)$ is a point $[B]\in\oa[1]\subset\Sat$. If $[B]=[i\infty]=\partial\oa[1]$, then the
surface $S_1$ is contained in $\beta_3=\osi{1+1+1}$, and then by \Cref{lm:E111} its class is a positive linear combination of $\osi{K3+1}$ and $\osi{C4}$.

If $[B]\in\calA_1$, then $S_1\in\oE(\pi^{-1}([B])$, and thus we want to determine the generators of $\oE(\pi^{-1}([B])$.
For this we recall the geometry of this fiber, see \cite[Section 3]{ergrhu2}: up to a finite group action the normalization of this fiber is the total space of the $\PP^1$-bundle $\PP(\calO\oplus P)$ over $B\times B$. The map
to $\oa$ identifies the $0$-section and the $\infty$-section. Note that there is a $\CC^*$-action on this $\PP^1$-bundle, which fixes the $0$- and the $\infty$-sections of the bundle pointwise, and acts by multiplication on the $\CC^*$-fibers.
This action descends to $\pi^{-1}([B])$.
Similarly to the proof of \Cref{lm:E111}, one can, by the Borel fixed point theorem, choose as generators of the cone $\oE(\pi^{-1}([B])$ irreducible surfaces that are setwise fixed by the $\CC^*$-action.

Since the orbits of $\CC^*$ are either points in the $0$-section (identified with the $\infty$-section), or entire fibers of the bundle, it follows that generators of $\oE(\pi^{-1}([B])$ can be chosen to be surfaces that
either are contained in the~$0$-section, or as surfaces that contain fibers of the $\PP^1$-bundle.

If $S_i$ is contained in the $0$-section of this bundle, we first recall that this~$0$-section, as a subvariety of~$\oa$, is identified with $\osi{K3}$. Since the surfaces~$S_D$ and~$\osi{K3+1}$ are contained in~$\osi{K3}$ and their homology
classes are not proportional by \Cref{prop:tildeY}, it follows that their fundamental classes generate~$H_4(\osi{K3})$, which is two-dimensional by \Cref{lm:cohY}.
Thus we can write (the class of) $S_i$ as a linear combination $S_i=aS_D+b\osi{K3+1}$ with some coefficients $a$ and $b$. Since $LM.S_i=LM.\osi{K3+1}=0$ while $LM.S_D>0$, it follows that $a=0$. But then $S_i$ is proportional to $\osi{K3+1}$, and since these are both effective surfaces, the coefficient of proportionality must be positive.

Now it remains to classify irreducible effective surfaces in $\pi^{-1}([B])$ that contain entire $\PP^1$-fibers. The intersection of any such surface with the $0$-section is an effective curve $C$ in $B\times B$. Of course it would now suffice to compute the cone of effective surfaces on $B \times B$ for any elliptic curve $[B]\in\calA_1$, taking into account the further invariance under the finite automorphism of the group. However, instead we will consider $C$ as an effective curve $C\in\Eff[1](\osi{K3})$. To determine the cone of pseudoeffective curves $\oE[1](\osi{K3})$, we will determine the dual cone $\Nef[1](\osi{K3})$, or rather its pullback to $\wub$.
Before we discuss this in detail, we recall that $\si{K3}$ is a finite quotient of $\ua[1] \times_{\ab[1]} \ua[1]$ and that $\osi{K3}$ is a finite quotient of $\wub$,  see \Cref{lm:osiK3geometry}.
Moreover $\si{1+1}$ is a finite quotient of the $\CC^*$-bundle over $\si{K3}$ which is given by the Poincar\'e bundle with the $0$-section removed. Further recall  from \Cref{lm:cohY} that $\dim H^2(\osi{K3})=2$.
The pullbacks of the classes $L$ and $D$ to $\osi{K3}$ are clearly independent, and therefore the same holds for the pullbacks of $L$ and $M$.
Since the classes $L$ and $M$ are nef on $\oa$, their restrictions to $\osi{K3}$ and pullbacks to $\wub$ are also nef. We claim that neither $L$ nor $M$ is ample on~$\osi{K3}$.

Indeed, consider any effective curve $C_{\alpha}\subset\si{K3}\subset\osi{K3}$ which is contained in a fiber of the natural projection $\ua[1] \times_{\ab[1]} \ua[1] \to \ab[1]$. Then $L.C_{\alpha}=0$ and thus~$L$ is not ample on~$\osi{K3}$.
Similarly, the curve $C_{\beta}=\oa[1]\times [i\infty]$ contained in the $0$-section of $\wub$ (where we as always think of the point $[i\infty]$ as the boundary $[i\infty]=\partial\oa[1]$) satisfies $M.C_{\beta}=0$, and thus $M$ is not ample on $\osi{K3}$, either.
Thus we have found two nef non-ample divisor classes in $H^2(\osi{K3})=\RR^{\oplus 2}$, proving that $\Nef[1](\osi{K3})=\RR_{\ge 0}L+\RR_{\ge 0}M$.

Dually, it follows that $\oE[1](\osi{K3})=\RR_{\ge 0} C_{\alpha}+\RR_{\ge 0} C_{\beta}$, since these are two effective curve classes that have zero pairing with the two generators of the cone of nef divisors (equivalently, they are contracted under the corresponding morphisms). We now consider the pre-images of the (open) curves~$C_\alpha$ and~$C_\beta$ in the $\CC^*$-bundle which is the Poincar\'e bundle with the $0$-section removed and take the closure of their images in $\osi{1+1}$. Call these surfaces~$S_{\alpha}$ and~$S_{\beta}$ respectively. It now follows that the cone of effective surfaces in~$\osi{1+1}$ that contain the entire closures of images of the~$\CC^*$-fibers is generated by the classes of the surfaces~$S_\alpha$ and~$S_\beta$. Since our original surface~$S_i$ had to be contracted to a point in~$\Sat$, the restriction~$L|_{S_i}$ is the trivial bundle. Writing the class as $S_i=aS_\alpha+bS_\beta$, this implies that $b=0$, and thus $S_i=aS_\alpha$, and furthermore $a>0$ since both $S_i$ and $S_\alpha$ are effective surfaces. To compute the class of $S_\alpha$, for $C_\alpha$ we could choose for example the curve $B\times \{0\}$, but then deforming $B\in\calA_1$ to $[i\infty]=\partial\oa[1]$ deforms $S_\alpha$ to an effective (and not necessarily irreducible) surface contained in $\beta_3=\osi{1+1+1}$, and thus finally \Cref{lm:E111} applies to show that the homology class of $S_\alpha$ is a positive linear combination of $\osi{C4}$ and $\osi{K3+1}$, again.
\end{proof}

To prove extremality of the other surfaces, we will use the following observation of Blankers:
\begin{lm}[Blankers {\cite[Cor.~2.8]{blankers}}, setting $Y=\Exc(f)$]\label{lm:blankers}
If $f:X\to W$ is a projective morphism, and the class of the subvariety $Z\subseteq \Exc(f)$ is extremal effective, and $Z$ is contracted by $f$, then the class of $Z$ is extremal effective on $X$.
\end{lm}

To prove extremality of the surfaces $S_A$ and $S_F$, we will apply this result for the morphism given by a sufficiently high multiple $|nM|$ of the linear system $M$, where Shepherd-Barron's result ( \Cref{thm:sb} above) shows that the exceptional locus is $\oa[1]\times\oa[2]$, and we have investigated its geometry in \Cref{sec:oa1oa2}.
\begin{prop}\label{prop:SLS0extremal}
The surfaces $S_A$ and $S_F$ are extremal effective on $\oa$.
\end{prop}
\begin{proof}
Let $N: \oa[1] \times \oa[2] \to \oa$ be the natural map.
We want to use \Cref{lm:blankers} applied to $Z=N(\oa[1] \times \oa[2])$. For this we first have to check that they the surfaces $S_A$ and $S_F$ are contracted under the map $|mM|$.
The preimage of $S_F$ is the surface $S_{AF}$, whereas  the preimage of $S_A$ consists of~$S_{AA}$ and~$S_{DA}$.
Since $N^*M=M_2$ on $\oa[1] \times \oa[2]$ is in particular trivial on the first factor, the linear system $|mN^*(M)|$ clearly contracts any surface of the form $\oa[1]\times C$, and thus the surfaces $S_{AF}$ and $S_{AA}$.
Similarly, the exceptional locus of $M_2$ on $\oa[2]$ is  $\oa[1]\times\oa[1]$ and this also shows that $S_{DA}$ is  contracted under $|mN^*(M)|$. Hence $S_A$ and $S_D$ are contracted under $|mM|$.

It remains to show that the surfaces $S_A$ and $S_F$ are extremal effective on~$Z$. Assume that this is not the case for $S_F$. But then its preimage~$S_{AF}$ on $\oa[1] \times \oa[2]$ would not be extremal effective, contradicting
\Cref{lm:intersectionsoa1oa2}. For~$S_A$ we find by the same Lemma that any decomposition of $S_A$ would consist of a combination of surfaces of the class of $S_{DA}$ and $S_{AA}$. Since these surfaces both map to
$S_A$ we can conclude that $S_A$ is also extremal effective.
\end{proof}
As we have observed before, while on $\oa[1]\times\oa[2]$ all three surfaces $S_{AA}$, $S_{AF}$  and $S_{DA}$ are of course distinct, the images $N(S_{AA})$ and $N(S_{DA})$ are the same surface $S_A\subset\oa$, as one considers the product abelian threefold, without marking the factors --- so we obtain only two extremal effective surfaces on~$\oa$.

Furthermore, it is tempting to prove that the image $S_D$ of the surface~$S_{DD}$, which we have also shown to be extremal on $\oa[1]\times\oa[2]$, is also extremal effective on $\oa$, but as it is not contracted by $|mM|$, the above argument does not apply.
To deal with $S_D$, we will thus apply Blankers' result \Cref{lm:blankers}   to
the morphism $\pi:\oa\to\Sat$ given by a sufficiently high multiple of the linear system~$|nL|$. Since~$L$ is ample on~$\Sat$, the exceptional locus $\Exc(\pi)$ is equal to $\partial\oa$. We will thus apply \Cref{lm:blankers}
to extremal effective surfaces in $\partial\oa$ --- if they are contracted in $\Sat$, they must be extremal in $\oa$. This is where we are going to use the special feature of the geometry of genus $3$, that there is a
finite map $f:\ox\to\partial\oa[2]$, as discussed in \Cref{sec:ox1ox1}.

\begin{prop}\label{prop:S0extremal}
The surface $S_D$ is an extremal effective ray of $\Eff(\oa)$.
\end{prop}
\begin{proof}
We apply \Cref{lm:blankers} to the contraction $\pi:\oa\to\Sat$, whose exceptional locus is $\partial\oa$. Since $\pi(S_D)$ is a curve, isomorphic to $\oa[1]$, in $\Sat$, it suffices to prove that $S_D$ is an extremal effective rays of $\Eff(\oa)$ (while we note that this argument will work for $S_F$, but cannot be applied to $S_A$ that is mapped isomorphically to its image $\oa[1]\times\oa[1]\subset\Sat$ under $\pi$).

Recall now that there is a finite morphism $f:\ox\to\partial\oa$ and that $S_D$ parameterizes semi-abelic threefolds which are the product of a fixed elliptic curve with a singular semi-abelic surface.
Suppose for contradiction that~$S_D$ is not an extremal effective ray of $\Eff(\partial\oa)$ -- but then its preimage, which is an irreducible surface on $\ox$, would not be an extremal effective
ray of $\Eff(\ox)$ either.
Thus finally it suffices to prove that the preimage $f^*S_D$ on~$\ox$ is an extremal effective ray of $\Eff(\ox)$. This preimage is contained in $V=\ox[1]\times\ox[1]/S_2$, and we think of it as $S_3\subset V\subset\ox$.

The advantage of thinking of $S_3\subset\ox$ as opposed to $S_D\subset\partial\oa$ is that we can use the morphism $p:\ox[2]\to\oa[2]$, and not only $\pi|_{\partial \oa}\circ f:\ox\to\Sat[2]$.
In particular the divisor class $M_2$ is nef on~$\oa[2]$, and thus $p^* M_2$ is nef on $\ox[2]$. Since on the decomposable locus $\oa[1]\times\oa[1]\subset\oa[2]$ the class $M_2$ restricts to the direct sum of the pullback of $M_1$ from each factor, and $M_1=12L-D$ on $\oa[1]$ is trivial, it follows that the pullback $p^*M_2|_V$ is also trivial. In particular $p^*M_2|_{S_3}$ and $p^*M_2|_{S_4}$ are both trivial bundles (crucially, we note that $M_3|_{\partial\oa}\ne M_2$, which is what makes this argument work, while $M_3$ is non-zero on $S_D$, as can be seen from the intersection numbers $S_D.LM$ and $S_D.M^2$ being non-zero in \Cref{prop:intersections}).

Suppose that $S_3=\sum a_i S_i$, with $a_i>0$ and $S_i\subset \ox$ some irreducible effective surfaces. Let $T$ be the class of the theta divisor on $\ox$, trivialized along the $0$-section, so that $T$ is a nef class on $\ox$, which is furthermore $p$-ample. Since $p^*M_2|_V$ is trivial, it follows that $p^*M_2|_{S_3}$ is the trivial bundle, and thus the intersection numbers $S_3.(p^*M_2)^2=S_3.(p^*M_2)T=0$ both vanish. Furthermore  $T|_V=T_1+T_2$ and thus $T^2|_V=2T_1T_2$. Recalling the intersection numbers of $S_3=L_1Z_1+L_2Z_2$ on $V$, it follows that $S_3.T^2=0$. Since $p^*M_2$ and $T$ are nef divisors on $\ox$, the three classes $(p^*M_2)^2,p^*M_2T$, and $T^2$ are nef, and since they have zero intersection with $S_3$, it follows that they also have zero intersection with each of the effective surfaces $S_i$:
$$
 S_i.(p^*M_2)^2=S_i.p^*M_2T=S_i.T^2=0.
$$

If for some $i$ the image $p(S_i)$ is a surface, then $0=S_i.(p^*M_2)^2=p(S_i).M_2^2$ implies that $p(S_i)$ must be contained in the exceptional locus $\Exc(M_2)\subset\oa[2]$, which by the result of Shepherd-Barron, \Cref{thm:sb}, is equal to
$N(\oa[1]\times\oa[1])$. Thus $p(S_i)\subset N(\oa[1]\times\oa[1])$ and $S_i\subset V$.

If for some $i$ the image $p(S_i)$ is a curve $C\subset\oa[2]$, then since $T$ is $p$-ample, the equality $S_i.(p^*M_2)T=0$ implies $C.M_2=0$, and thus again $C$ must be contained in  $\Exc(M_2)=N(\oa[1]\times\oa[1])$,
and thus again $S_i\subset V$.

Finally, if $p(S_i)$ were a point in $\oa[2]$, then since the class $T$ is $p$-ample, it would follow that $S_i.T^2>0$, which is a contradiction.

Thus if $S_3=\sum a_i S_i$, for $S_i\subset \ox[2]$, it follows that each $S_i$ is contained in~$V$. Since by \Cref{prop:effV} the surface~$S_3$ is extremal effective on~$V$, it thus follows that~$S_3$ is extremal effective on $\partial\oa[3]$.
\end{proof}
\begin{rem}
The proof above applies verbatim to the surface $S_F$ as well. Indeed, the only statements used about $S_D$ were that $S_D$ were contained in~$V$ (which implies $M_2|_{S_D}=\calO_{S_D}$), and that $S_D.T^2=0$. Both of these hold for $S_F$, and thus we get an alternative proof of the extremality of $S_F$ in this way. Note that for the surface $S_1\subset V\subset\ox$ we also have $M_2|_{S_1}$ trivial, and the only part of the above argument that fails is the statement that $S_1.T^2=0$. Indeed, we know that $S_1$ is not an extremal effective ray on $\oa$, as it is a positive linear combination of $\osi{C4}$ and $\osi{K3+1}$, as discussed above.
\end{rem}

\section{A surface $S_P\subset\osi{1+1}$ homological to $S_F$}\label{sec:SF}
While we have constructed our five surfaces in an easy way as images of surfaces contained in $\oa[1]\times\oa[2]$ and in~$V$, it is very natural to look for further effective surfaces contained deeper in the boundary of~$\oa$.
In \Cref{sec:osiK3} we studied in detail the geometry of the three-dimensional toroidal stratum $\osi{K3}$.
\Cref{lm:E111} determines the effective cone $\Eff(\osi{1+1+1})$ of the other three-dimensional
toroidal stratum $\osi{1+1+1}$, and we see that no further extremal effective surfaces can be constructed there. We now look at the unique four-dimensional toroidal stratum $\osi{1+1}$. While it seems possible to determine $\Eff(\osi{1+1})$
with our methods with more work, here we only define one natural geometric surface contained in $\osi{1+1}$.  We cannot see immediately why this surface should  be homologous to one of our five extremal effective surfaces.
However, we compute its homology class, which turns out to be equal to the class of $S_F$.

Roughly speaking,  we define $S_P$ as the image in $\osi{1+1}$ of the closure of the total space of the $\PP^1$-bundle $\PP(\mathcal O \oplus P)$  over the $0$-section $\calA_1\to\calX_1\times_{\calA_1}\calX_1$ where $P$ denotes the Poincar\'e bundle.
To make this more precise we start with the cone $\langle x_1^2,x_2^2 \rangle$. The corresponding toroidal variety $\calT(\si{1+1})$ is a $\CC^*$-bundle over  $\ua[1] \times_{\ab[1]} \ua[1]$. Geometrically this is the total space of the Poincar\'e bundle $P$ with the $0$-section removed. We restrict this to the $0$-section $0: \ab[1] \to \ua[1] \times_{\ab[1]} \ua[1]$ and denote this restriction by~$T_P^0$. Since the Poincar\'e bundle is trivial over the $0$-section we thus obtain the trivial $\CC^*$-bundle over~$\ab[1]$.  We define $S_P^0$ to be the image of this trivial $\CC^*$-bundle under the map $\calT(\si{1+1}) \to \si{1+1} \subset \oa$ that maps onto the open stratum of~$\beta_2$.
This map is given by the quotient under the group $G(\si{1+1}) \subset \Sp(4,\ZZ)$ that is generated by the involutions
\begin{equation*}
\begin{array}{lll}
(x_1,x_2) &\mapsto& (\pm x_1, \pm x_2) \\
(x_1,x_2) &\mapsto& (x_2,x_1)\,.
\end{array}
\end{equation*}
We note here that in the end we also have to take care of the extra factor $1/2$ which comes from the discussion in \Cref{rem:stacky12}. The first two involutions define the Kummer involutions on the fibers of $\ua[1] \times_{\ab[1]} \ua[1]$, the second involution interchanges the two factors. However, we also note that the involutions $(x_1,x_2) \mapsto (- x_1, x_2)$
and  $(x_1,x_2) \mapsto (x_1, - x_2)$ both act as the involution $z\mapsto 1/z$ on the fibers of $T_P^0$. Hence the map $T_P^0 \to S_P^0$ is a $2:1$ map that branches over the sections $\ab[1] \times \{ \pm 1\}$.
Since the quotient of~$\CC^*$ by the involution  $z\mapsto 1/z$ is isomorphic to $\CC$, it follows that the quotient $S_P^0$ is isomorphic to the trivial $\CC$-bundle over $\ab[1]$. The surface $S_P$ is  defined as the closure of $S_P^0$ in $\beta_2$ or, equivalently, in $\oa$. We note however, that $S_P$ is not a normal surface. Indeed, the closure of any $\CC^*$ over a point in $\ab[1]$ is a nodal cubic.

We set $T_P:= \PP^1 \times \PP^1$ and think of this as a compactification of $T_p^0=\ab[1] \times \CC^*$.
\begin{lm}
The surface $S_P$ is isomorphic to $\oa[1] \times \PP^1$ and the map $T_P^0 \to S_P^0$ extends to a $2:1$ map $T_P \to S_P$.
\end{lm}
\begin{proof}
The compactification of $S_P^0$ is done in two steps. Firstly, we compactify the fibers $\CC$ of the product $\ab[1] \times \CC$  to $\PP^1$. This is given by the inclusion of the cones $\si{1+1} \subset \si{K3}$.  Secondly, we have to extend the surface
over the cusp of $\ab[1]$. A straightforward calculation in toric geometry shows that we  add a copy of $\CC$ coming from the inclusion $\si{1+1} \subset \si{1+1+1}$ and a point given by the inclusion  $\si{1+1} \subset \si{K3+1}$ and that the resulting surface is smooth.
Hence we obtain a $\PP^1$-bundle over $\oa[1]$, which must be a Hirzebruch surface. We claim this is $\PP^1 \times \PP^1$. This can be seen e.g.~since the section at $\infty$ as well as the sections $\ab[1] \times \{ \pm 1\}$
extend to three disjoint sections on the $\PP^1$-bundle. This can only happen for $\PP^1 \times \PP^1$. The rest of the claim is then obvious.
\end{proof}

\begin{rem}
We could also have worked with  the surface $S_P$ directly, but our computation of intersection numbers is more consistent with our other computations if we do this on $T_P$.
\end{rem}

\begin{rem}
Another approach to prove that $S_P$ is the trivial $\PP^1$-bundle is to work with a level-$n$ cover. Then we obtain a $\CC^*$-bundle in $\oa$ which is compactified to a $\PP^1$-bundle. In this case there is a global involution given by $(x_1,x_2) \mapsto (x_1,-x_2)$, which interchanges the $0$-section and the $\infty$-section, thus showing that they have the same self-intersection.
\end{rem}

We can now compute the intersection numbers of $S_P$ with a basis of $H^4(\oa)$.

\begin{prop}\label{prop:interectionumbersSP}
The surface $S_P$ has the following intersection numbers:
\begin{equation*}
L^2.S_P=0, \, LM.S_P= \frac{1}{96}, \, M^2.S_P=0, \,  \beta_2.S_P= - \frac{1}{8}\,,
\end{equation*}
\begin{equation*}
  LD.S_P= - \frac{1}{96}, \, D^2.S_P= - \frac{1}{4}\,.
\end{equation*}
\end{prop}
\begin{proof}
As before, we will work on a level cover $\oa(n)$, and then divide by the order of the stabilizer subgroup of $\Sp(6,\ZZ/n\ZZ)$ fixing the surface $S_P(n)$.
The preimage of $S_P$ in $\oa(n)$
consists of several components, each of which is isomorphic to $\oa[1](n) \times \PP^1$,
where $\oa[1](n)$ is the $0$-section of  the fiber product $\ox[1](n) \times_{\oa[1](n)} \ox[1](n)  \to \oa[1](n)$. We pick one such  component and denote it by~$S_P(n)$. We shall further
denote the (class of) the $0$-section (or equivalently the $\infty$-section) by $S$, whereas the class of a fiber will be denoted by $N$.

We first want to understand the restriction of the boundary $D(n)$ of $\oa(n)$ to $S_P(n)$. Since $S_P$ is contained in $\beta_2$, the surface $S_P(n)$ is contained in two boundary components which we denote $D_i$ and $D_j$.
We can assume that they correspond to $\langle x_1^2 \rangle $ and $\langle x_2^2 \rangle$ respectively.
The $0$-section and the $\infty$-section are cut out by two further boundary components, which we denote by $D_{i+j}$ and
$D_{i-j}$ and they correspond to the cones $\langle (x_1 \pm x_2)^2 \rangle $. The fibers of~$S_P(n)$ over the cusps of $\oa[1](n)$ are cut out by further boundary components $D_k$. We also recall that the covering map $\oa(n) \to \oa$ is branched of order $n$ along all boundary components.

We recall that the restriction of each boundary component to itself is equal to $D_i|_{D_i}=- (2/n)\Theta$ (see \cite[Prop.~5.1]{ergrhu2}), where $\Theta$ is the theta divisor trivialized along the $0$-section. However, we also note that the theta
divisor $\Theta$ has degree $n$ on the fibers of the surface $S_P$.
This implies that
\begin{equation}\label{equ:intersectionSP11}
D_i|_{S_P(n)} = - 2S, \quad i=1,2\,.
\end{equation}
Immediately from the geometric situation we conclude that
\begin{equation}\label{equ:intersectionSP21}
D_{i \pm j}|_{S_P(n)} = S\,
\end{equation}
and finally
\begin{equation}\label{equ:intersectionSP31}
\sum _{\{k \neq i,j, i \pm j\}} D_k|_{S_P(n)} = nt(n)N\,.
\end{equation}
We further have
\begin{equation*}
12L|_{S_P(n)} = nt(n)N\,.
\end{equation*}
This shows that
\begin{equation*}
L(nD(n))|_{S_P(n)} = \frac{1}{12}nt(n)N(-2nS + nt(n)N)=- \frac16 n^2t(n)\,.
\end{equation*}
Here we have taken the intersection with respect to $nD(n)$, in order to take into account the branching of order $n$ of the cover $\oa(n) \to \oa$ along the boundary.

We now must determine the order of the stabilizer subgroup $G(S_P(n))$ of~$S_P(n)$ in $\Sp(6,\ZZ/n\ZZ)$. We have the following elements. First we have the involutions $x_i \mapsto \pm x_i, \, i=1,2$, as well as the involution which
interchanges~$x_1$ and~$x_2$. The involution $(x_1,x_2) \mapsto (-x_1,-x_2)$ acts trivially on~$S_P(n)$ whereas the involutions $(x_1,x_2) \mapsto (-x_1,x_2)$ and  $(x_1,x_2) \mapsto (x_1,-x_2)$ both act by $z\mapsto 1/z$ on the fibers and, in particular, identify the boundary divisors cut out by $D_i$ and $D_j$.
Next we observe that the map $S_p(n) \to S_P$ further factors through the map $z \mapsto z^n$ on the $\PP^1$ fibers. This gives rise to  the $n$-fold branching along the boundary components $D_i$ and $D_j$.
Finally we have the action of the group
$\SL(2,\ZZ/n\ZZ)$, which also induces the $n$-fold  branching along the boundary components which cut out the fibers over the cusps. Note that due to $(x_1,x_2) \mapsto (-x_1,-x_2)$ and the involution $-1 \in \SL(2,\ZZ/n\ZZ)$ we have now also taken care of the involution $-1 \in \Sp(6,\ZZ/n\ZZ)$ (cf. \Cref{rem:stacky12}).
This gives us
\begin{equation*}
|G(S_P(n))| =  2^2\cdot  2 \cdot n \cdot |\SL(2,\ZZ/n\ZZ)| = 16n^2t(n)\,.
\end{equation*}
Altogether this gives
\begin{equation*}
LD.S_P= \frac{-1/6 n^2t(n)}{16n^2t(n)}= - \frac{1}{96}\,.
\end{equation*}
Since $L^2=0$ on $\oa[1]$ we also have
\begin{equation*}
L^2.S_P= 0.
\end{equation*}

The other intersection numbers are now also easy to compute. Using the fact that the $\PP^1$-bundle $S_P$ is trivial, we have $S^2=0$, and hence
\begin{equation*}
D^2.S_P=\frac{(-2nS + nt(n)N)^2}{|G(S_P(n))|}=\frac{-4n^2t(n)}{16n^2t(n)}= -\frac{1}{4}\,.
\end{equation*}
Similarly $ML.S_P=(12L-D)L.S_P$ is given by
\begin{equation*}
ML.S_P=\frac{(nt(n)N  + 2nS -nt(n)N)(\frac{1}{12}nt(n)N)}{|G(S_P(n))|}= \frac{1}{6}\cdot \frac{1}{16}=\frac{1}{96}.
\end{equation*}
Next we can compute $M^2.S_P=(12L-D)^2.S_P$ as
\begin{equation*}
M^2.S_P=\frac{(nt(n)N+2nS - nt(n)N)^2}{|G(S_P(n))|}=0.
\end{equation*}

It remains to compute $\beta_2.S_P$. On $\oa(n)$ we have that
\begin{equation*}
\beta_2(n)= \sum_{i < j}D_i(n)D_j(n)\,.
\end{equation*}
If we fix one component $S_P(n)$, then the boundary components which are relevant for us are $D_i,D_j,D_{i\pm j}$ and the $D_k$ cutting out the fibers over the cusps of $\oa[1](n)$. By our discussion above, namely \eqref{equ:intersectionSP11}, \eqref{equ:intersectionSP21}, and \eqref{equ:intersectionSP31}, the only non-zero intersections occur when one
factor from $D_i,D_j,D_{i \pm j}$ is paired with some~$D_k$. More precisely we obtain
\begin{equation*}
\beta_2(n).S_P(n)=\sum_s(D_i+D_j+D_{i-j} + D_{i+j})D_k=-2t(n)\,.
\end{equation*}
Taking the branching of $\oa(n) \to \oa$ of order $n$ along the boundary into account we thus find
\begin{equation*}
\beta_2.S_P=\frac{-2n^2t(n)}{16n^2t(n)} = - \frac18\,.
\end{equation*}
\end{proof}
\begin{cor}\label{cor:SF=SP}
The classes of the surfaces $S_F$ and $S_P$ in  $\Eff(\oa)$ coincide, i.e. $[S_F]=[S_P]$.
\end{cor}
\begin{proof}
This is an immediate consequence of the calculation of the class of $S_F$ in \Cref{prop:intersections} and the intersection numbers for $S_P$ computed above.
\end{proof}

\section{Further questions and directions}\label{sec:further}
In this section we make a number of remarks about various further directions one could pursue, and partial progress that can be made.
\begin{rem}[Further evidence for \Cref{conj:eff}]\label{rem:evidence}
While we were not able to prove that the surfaces $S_F$, $S_A$, $S_D$, $\osi{K3+1}$, $\osi{C4}$ generate $\oE(\oa)$, we verified the validity of some properties that would hold if this were the case --- which thus gives further evidence for our main \Cref{conj:eff}. First, recall that the conjecture is equivalent to the dual \Cref{conj:nef} that the cone $\Nef(\oa)$ is generated by $L^2,LM,M^2,F_1,F_2$. The first three of these classes are of course known to be nef, and thus the conjecture is equivalent to the statement that $F_1$ and $F_2$ are nef. The top intersection number of any collection of nef classes of arbitrary dimensions is always non-negative. Using van der Geer's~\cite{vdgeerchowa3} determination of the ring structure on the Chow ring $CH^*(\oa)$ and the result of~\cite{huto1} that $H^*(\oa)=CH^*(\oa)$, all top intersection numbers of these five classes on $\oa$ can be computed, and we verified, with Maple, that they are indeed all non-negative.
Numerical experiments show that even small variations of the intersection numbers in~\eqref{eq:intersect} easily produce classes with negative intersection numbers. This shows that this check is numerically non-trivial.

Secondly, one could try to construct further effective surfaces in $\oa$ and check whether their classes lie in the cone spanned by our five extremal effective surfaces. To this end, notice that we have constructed the surface~$S_P$, which turns out to have the same class as surface $S_F$. Additionally, our computations above show that the class of a generic fiber of the map $\pi|_{\beta_{1}}:\beta_{1}\to\calA_2$
is equal to a positive linear combination of $\osi{C4}$ and $\osi{K3+1}$. In addition we checked that the class of the surface described geometrically as the family of theta divisors on the Kummer surfaces that are the fibers
of $\pi:\oa\to\Sat$ over any curve $C\subset\calA_2$ also lies in our cone.

Finally, our computations show that the pseudoeffective cones of $\oa[1]\times\oa[2]$ and of $V=\pi^{-1}(\oa[1]\times\oa[1])$ are generated by suitable subsets of our surfaces --- and thus any effective surface contained in $\oa[1]\times\oa[2]\cup V$ has a class contained in the cone spanned by our five surfaces. In the next remark we discuss the relevance of $\oa[1]\times\oa[2]\cup V$ from the point of view of thinking of intersections of effective surfaces with effective divisors on~$\oa$.
\end{rem}
\begin{rem}[Towards proving extremality of some surfaces]\label{rem:oa1oa2v}
We note that the cone of effective divisors on~$\oa$ is known: $\oE[1](\oa)=\RR_{\ge 0}D+\RR_{\ge 0}\Hyp_3$, where  $\Hyp_3$ is the closure of the  locus of Jacobians of hyperelliptic curves. Equivalently, one can think of this as
abelian threefolds with a vanishing theta-null, that is the zero locus of the modular form $\Theta_{\rm null}:=\prod\theta_m^8$, where the product is taken over all theta constants with even characteristics.
This shows that its class is given by $\Hyp_3=9L-D$.
We now make the following general observation. Suppose $S$ in an effective surface on some variety $X$, and suppose $s$ is some section of a line bundle $L$ on $X$, whose zero locus
is a divisor $D\subset X$. If $S\not\subseteq D$, then the intersection $C=S\cap D$ is a curve whose homology class is equal to $S.L$. Suppose now $L$ is an ample line bundle on $X$ and $E$ is an effective divisor on $X$. Then if the intersection number $S.LE<0$, this means that for any section $s$ of $L$ the intersection $C=S\cap \{s=0\}$ must be contained in~$E$, which implies that $S\subset E$.

Applying this to $\oa$, where $\Nef[1](\oa)=\RR_{\ge 0}L+\RR_{\ge 0}M$, we see that if for an effective surface $S$ one has $S.M\Hyp_3<0$, then $S$ is contained in $\Hyp_3$.
Moreover, it is known that the moving slope of~$\oa$ is equal to $28/3$, which is to say that $H^0(\oa,n(28L-3D))=H^0(\oa,n(3M-8D))$
for $n$ large enough contains irreducible divisors, which thus do not contain $\Hyp_3$. The simplest example of such a divisor is the zero locus of the modular form $F_\theta:=(\sum \theta_m^{-8})\prod\theta_m^8$, i.e.~the sum of products of all but one theta constants (the eighth powers here and in the definition of the theta-null divisor are to avoid eighth roots of unity in the theta transformation formula). As was shown by Igusa~\cite{igusadesing} and Tsushima~\cite{tsushima}, the set-theoretic intersection of the zero loci of $\Theta_{\rm null}$ and $F_\theta$ in $\oa$ is equal precisely to $N(\oa[1]\times\oa[2])\cup V$
(actually with the two components appearing with multiplicities $240$ and $10$, see proof of~\cite[Prop.~3.2]{vdgeerchowa3}).
This implies that any effective surface $S\subset\oa$ such that $S.M(3M-8L)<0$ must be contained in the union of $N(\oa[1]\times\oa[2])\cup V$, and thus its class must be a positive linear combination of the generators of the effective cones there, which we have determined.

We note that $S_D.M(M-2L)=0$, and thus it is tempting to extend this analysis further by analyzing the base locus of the linear system $n|M-2L|=n|10L-D|$ on $\oa$. To this end, we let $P$ be the product of the minus eighth powers of the following 6 even theta constants with characteristics:
$$
 P:=\left(\tc{011}{011}\tc{011}{111}\tc{100}{000}\tc{100}{011}\tc{111}{000}\tc{111}{011}\right)^{-8},
$$
and define the following modular form
$$
 F_P:=\left(\sum_{\gamma\in \Sp(6,\ZZ/2\ZZ)} \gamma\circ P\right)\cdot\prod_m\theta_m^8,
$$
where the product is taken over all 36 even theta characteristics in genus 3, and we think of the action of the symplectic group $\Sp(6,\ZZ)$ on the set of even characteristics by affine-linear transformations, so that $\gamma\circ P$ is the product of minus eighth powers of the 6 even theta constants whose characteristics are the images of the above 6 under the action of $\gamma$.

By a careful computation with theta constants it is possible to show that~$F_P$ does not vanish identically either on $N(\oa[1]\times\oa[2])$ or on~$V$ (and we believe it to be the unique slope 10 modular form with this property). Thus if a surface $S\subset \oa$ satisfies $S.M(M-2L)<0$, then $S$ is contained in the zero locus of $F_P$ intersected with $N(\oa[1]\times\oa[2])\cup V$.

While this allows us to further reduce the possibilities for the classes of effective surfaces on~$\oa$, it still does not allow us to fully determine the effective cone, as we still have no way of dealing with surfaces $S$ such that $S.M(M-2L)>0$.
\end{rem}

\begin{rem}[Extremal versus extremal effective]\label{rem:pseudoeffeff}
Note that we have only proven that our five surfaces are extremal effective, i.e. that they are generating rays of $\Eff(\oa)$, but not that they are extremal, i.e.~not that they are generating rays of $\oE(\oa)$. While we were able to fully determine $\oE(\oa[1]\times\oa[2])$ and $\oE(V)$, and not just the effective cones, this does not immediately yield extremality of the surfaces $S_A,S_D,S_F$. However, since essentially our proofs that these surfaces on~$\oa$ are extremal effective proceed by arguing that they are contracted under morphisms defined by high multiples of~$L$ and~$M$, it is plausible that a little more can be said using the results of \cite{fulgerlehmann} on images of pseudoeffective classes under contracting morphisms. It does not seem to us that this would suffice to prove extremality of all our five surfaces.
\end{rem}

\begin{rem}[Higher dimension and higher genus]\label{rem:higherdim}
By~\cite[Lem.~2.6]{blankers}, we know that if $Z\subseteq\Perf[i]$ is extremal effective, then $Z\times \Perf[g-i]$ is extremal effective on $\Perf[i]\times\Perf[g-i]$.
Thus starting from the knowledge of (extremal) effective curves on $\Perf[i]$ by the results of~\cite{shepherdbarron}, and from the knowledge of the cones of effective divisors on $\Perf$ for low genus, one can construct
extremal effective classes of various dimensions on $\Perf[i]\times\Perf[g-i]$. Then recall that for any $0<i<g$ there is a product map $N:\Perf[i]\times\Perf[g-i]\to\Perf[g]$, which is a normalization of the image
(if $i\neq g-i$ and the $2:1$ map to the symmetric product otherwise), and
thus it is natural to ask whether the images under~$N$ of extremal effective cycles constructed as such products are extremal effective on $\Perf$. This happens to be the case for the images of the extremal effective surfaces under the map $N:\oa[1]\times\oa[2]\to\oa$ that we study. In general,
since $\Exc(M_g)=N(\oa[1]\times\Perf[g-1])$ for arbitrary $g$ by \Cref{thm:sb},
for extremal effective classes on $\oa[1]\times\Perf[g-1]$ that are contracted under $M_{g-1}$ (for example those that are a product of the form $\oa[1]\times Z$), it follows that their
images on $\Perf$ are also extremal effective. However, already for the map $N:\oa[2]\times\oa[2]\to\Perf[4]$ we do not know whether the images of extremal effective classes are extremal effective, as this is unrelated to the
exceptional locus of any semi-ample bundle on $\Perf[4]$. While it is possible to obtain a number of explicit results on extremality of higher-dimensional cycles this way, we did not pursue this systematically.

Recall also that taking the images under the map $N$ of the product of our surfaces with a fixed $B\in\Perf[g-3]$ defines similarly five effective surfaces contained in $\Perf[g]$ for any genus $g\ge 3$. The intersection numbers of these classes with a basis of $H^4(\Perf[g])$ are still the same, given by \Cref{prop:intersections}, for any $g\ge 3$. Thus our five surfaces generate homology in any genus, and it is natural to wonder whether they might always be extremal effective, or, more ambitiously, if they generate the cone of effective surfaces. While again some partial results can be obtained, note that in general there is no higher genus analog of the finite map $\ox\to\partial\oa$ that we used heavily, and moreover of course the homology of the closure of the boundary stratum $\osi{K3}$ (and of course its geometry) are much more complicated.
\end{rem}


\begin{thebibliography}{BKLV19}

\bibitem[BH85]{barthhulek}
W.~Barth and K.~Hulek.
\newblock Projective models of {S}hioda modular surfaces.
\newblock {\em Manuscripta Math.}, 50:73--132, 1985.

\bibitem[BKLV19]{bastiaetal}
F.~Bastianelli, A.~Kouvidakis, A.~Lopez, and F.~Viviani.
\newblock Effective cycles on the symmetric product of a curve, {I}: {T}he
  diagonal cone.
\newblock {\em Trans. Amer. Math. Soc.}, 372(12):8709--8758, 2019.
\newblock With an appendix by Ben Moonen.

\bibitem[Bla20]{blankers}
V.~Blankers.
\newblock Extremality of rational tails boundary strata in
  {$\overline{\mathcal{M}}_{g,n}$}.
\newblock 2020.
\newblock Preprint arXiv:2002.12403. European Journal of Mathematics, to appear.

\bibitem[CC15]{chcohighercodim}
D.~Chen and I.~Coskun.
\newblock Extremal higher codimension cycles on moduli spaces of curves.
\newblock {\em Proc. Lond. Math. Soc. (3)}, 111(1):181--204, 2015.

\bibitem[CFM13]{chfamo}
D.~Chen, G.~Farkas, and I.~Morrison.
\newblock Effective divisors on moduli spaces of curves and abelian varieties.
\newblock In {\em A celebration of algebraic geometry}, volume~18 of {\em Clay
  Math. Proc.}, pages 131--169. Amer. Math. Soc., Providence, RI, 2013.

\bibitem[CLO16]{coskunlesieutreottem}
I.~Coskun, J.~Lesieutre, and J.~C. Ottem.
\newblock Effective cones of cycles on blowups of projective space.
\newblock {\em Algebra Number Theory}, 10(9):1983--2014, 2016.

\bibitem[CT16]{chentarasca}
D.~Chen and N.~Tarasca.
\newblock Extremality of loci of hyperelliptic curves with marked {W}eierstrass
  points.
\newblock {\em Algebra Number Theory}, 10(9):1935--1948, 2016.

\bibitem[DSHS15]{DutourHulekSchuermann}
M.~Dutour~Sikiri\'{c}, K.~Hulek, and A.~Sch\"{u}rmann.
\newblock Smoothness and singularities of the perfect form and the second
  {V}oronoi compactification of {${\mathcal A}_g$}.
\newblock {\em Algebr. Geom.}, 2(5):642--653, 2015.

\bibitem[DSMS21]{sma6}
M.~Dittmann, R.~Salvati~Manni, and N.~Scheithauer.
\newblock Harmonic theta series and the {K}odaira dimension of {${\mathcal
  A}_6$}.
\newblock {\em Algebra Number Theory}, 15(1):271-- 285, 2021.

\bibitem[EGH06]{ergrhu1}
C.~Erdenberger, S.~Grushevsky, and K.~Hulek.
\newblock Intersection theory of toroidal compactifications of {${\mathcal
  A}_{4}$}.
\newblock {\em Bull. London Math. Soc.}, 38(3):396--400, 2006.

\bibitem[EGH10]{ergrhu2}
C.~Erdenberger, S.~Grushevsky, and K.~Hulek.
\newblock Some intersection numbers of divisors on toroidal compactifications
  of {${\mathcal A}_g$}.
\newblock {\em J. Algebraic Geom.}, 19:99--132, 2010.

\bibitem[FL16]{fulgerlehmann}
M.~Fulger and B.~Lehmann.
\newblock Morphisms and faces of pseudo-effective cones.
\newblock {\em Proc. Lond. Math. Soc. (3)}, 112(4):651--676, 2016.

\bibitem[GHT18]{grhuto}
S.~Grushevsky, K.~Hulek, and O.~Tommasi.
\newblock Stable cohomology of the perfect cone toroidal compactification of
  {$\mathcal A_g$}.
\newblock {\em J. Reine Angew. Math.}, 741:211--254, 2018.

\bibitem[GKM02]{gikemo}
A.~Gibney, S.~Keel, and I.~Morrison.
\newblock Towards the ample cone of {$\overline {\mathcal M}_{g,n}$}.
\newblock {\em J. Amer. Math. Soc.}, 15(2):273--294 (electronic), 2002.

\bibitem[GP21]{perna}
M.~Galeotti and S.~Perna.
\newblock Moduli spaces of abstract and embedded Kummer varieties.
\newblock {\em Internat. J. Math.} 32(8), 2021, Paper No. 2150054, 28 pp.


\bibitem[Gru09]{grAgsurvey}
S.~Grushevsky.
\newblock Geometry of {${\mathcal A}_g$} and its compactifications.
\newblock In {\em Algebraic geometry---{S}eattle 2005. {P}art 1}, volume~80 of
  {\em Proc. Sympos. Pure Math.}, pages 193--234. Amer. Math. Soc., Providence,
  RI, 2009.

\bibitem[GZ14]{grza}
S.~Grushevsky and D.~Zakharov.
\newblock The zero section of the universal semiabelian variety and the double
  ramification cycle.
\newblock {\em Duke Math. J.}, 163(5):953--982, 2014.

\bibitem[HKW93]{hukawebook}
K.~Hulek, C.~Kahn, and S.~Weintraub.
\newblock {\em Moduli spaces of abelian surfaces: compactification,
  degenerations, and theta functions}, volume~12 of {\em de Gruyter Expositions
  in Mathematics}.
\newblock Walter de Gruyter \& Co., Berlin, 1993.

\bibitem[HS04]{husaA4}
K.~Hulek and G.~Sankaran.
\newblock The nef cone of toroidal compactifications of {${\mathcal A}_4$}.
\newblock {\em Proc. London Math. Soc. (3)}, 88(3):659--704, 2004.

\bibitem[HT10]{huto1}
K.~Hulek and O.~Tommasi.
\newblock Cohomology of the toroidal compactification of {${\mathcal A}_3$}.
\newblock In {\em Vector bundles and complex geometry}, volume 522 of {\em
  Contemp. Math.}, pages 89--103. Amer. Math. Soc., Providence, RI, 2010.

\bibitem[Hul00]{huleknef}
K.~Hulek.
\newblock Nef divisors on moduli spaces of abelian varieties.
\newblock In {\em Complex analysis and algebraic geometry}, pages 255--274. de
  Gruyter, Berlin, 2000.

\bibitem[Igu67]{igusadesing}
J.-I. Igusa.
\newblock A desingularization problem in the theory of {S}iegel modular
  functions.
\newblock {\em Math. Ann.}, 168:228--260, 1967.

\bibitem[Mul20]{mullane1}
S.~Mullane.
\newblock On the effective cone of higher codimension cycles in
  {$\overline{\mathcal{M}}_{g,n}$}.
\newblock {\em Math. Z.}, 295(1-2):265--288, 2020.


\bibitem[Mul21]{mullane}
S.~Mullane.
\newblock Non-polyhedral effective cones from the moduli space of curves.
\newblock {\em Trans. Amer. Math. Soc.}, 374(9):6397--6415, 2021.


\bibitem[SB06]{shepherdbarron}
N.~Shepherd-Barron.
\newblock Perfect forms and the moduli space of abelian varieties.
\newblock {\em Invent. Math.}, 163(1):25--45, 2006.

\bibitem[SB16]{sbexceptional}
N.~Shepherd-Barron.
\newblock An exceptional locus in the perfect cone compactification of{
  ${\mathcal A}_g$}.
\newblock 2016.
\newblock Preprint arxiv:1604.05954.

\bibitem[Sch88]{schoen}
C.~Schoen.
\newblock On fiber products of rational elliptic surfaces with section.
\newblock {\em Math. Z.}, 197(2):177--199, 1988.

\bibitem[Sch15]{schaffler}
L.~Schaffler.
\newblock On the cone of effective 2-cycles on {$\overline{\mathcal M}_{0,7}$}.
\newblock {\em Eur. J. Math.}, 1(4):669--694, 2015.

\bibitem[Shi72]{shioda}
T.~Shioda.
\newblock On elliptic modular surfaces.
\newblock {\em J. Math. Soc. Japan}, 24:20--59, 1972.

\bibitem[Tai82]{tai}
Y.-S. Tai.
\newblock On the {K}odaira dimension of the moduli space of abelian varieties.
\newblock {\em Invent. Math.}, 68(3):425--439, 1982.

\bibitem[Tsu80]{tsushima}
R.~Tsushima.
\newblock A formula for the dimension of spaces of {S}iegel cusp forms of
  degree three.
\newblock {\em Amer. J. Math.}, 102(5):937--977, 1980.

\bibitem[vdG98]{vdgeerchowa3}
G.~van~der Geer.
\newblock The {C}how ring of the moduli space of abelian threefolds.
\newblock {\em J. Algebraic Geom.}, 7(4):753--770, 1998.

\end{thebibliography}

\end{document}